\documentclass[11pt, oneside]{article}   	
\usepackage{geometry}                		
\usepackage{graphicx}				
\usepackage{amssymb}
\usepackage{amsthm}
\usepackage{amsmath}
\usepackage{amsrefs}
\usepackage{mathrsfs}

\usepackage{color}   
\usepackage{hyperref}
\hypersetup{
    colorlinks=true, 
    linktoc=all,     
    linktocpage,
    linkcolor=magenta,  
}
\newcommand{\dsym}{LSym^{\delta}}
\newcommand{\dalg}{DAlg^{\delta}}
\newcommand{\calg}{DAlg}

\newcommand{\compdalg}{\widehat{\dalg}}
\newcommand{\compcalg}{\widehat{\calg}}

\newcommand{\env}{\text{Env}^{\Prism}}
\newcommand{\compenv}{\widehat{\text{Env}}^{\Prism}}
\newcommand{\compcalgenv}{\widehat{\text{Env}}}
\newcommand{\calgenv}{\text{Env}}
\newcommand{\grenv}{\text{Env}^{Gr}}

\newcommand{\fil}{F^{\geq 0}}

\newcommand{\gr}{Gr^{\geq 0}}

\newcommand{\pair}{F^{\{0,1\}}}

\newcommand{\dsymneut}{LSym^{\delta, [0]}}
\newcommand{\fdsym}{LSym^{\delta, \{0,1\}}}

\newcommand{\Z}{\mathbb{Z}}
\newcommand{\N}{\mathbb{N}}

\newcommand{\Perf}{\text{Perf}}
\newcommand{\Mod}{\text{Mod}}

\newcommand\rightthreearrow{%
        \mathrel{\vcenter{\mathsurround0pt
                \ialign{##\crcr
                        \noalign{\nointerlineskip}$\rightarrow$\crcr
                        \noalign{\nointerlineskip}$\rightarrow$\crcr
                        \noalign{\nointerlineskip}$\rightarrow$\crcr
                }%
        }}%
}

\theoremstyle{definition}

\newtheorem{theorem}{Theorem}[subsection]
\newtheorem{corollary}[theorem]{Corollary}
\newtheorem{lemma}[theorem]{Lemma}
\newtheorem{definition}[theorem]{Definition}
\newtheorem{observation}[theorem]{Observation}
\newtheorem{warning}[theorem]{Warning}
\newtheorem{remark}[theorem]{Remark}
\newtheorem{example}[theorem]{Example}
\newtheorem{proposition}[theorem]{Proposition}
\newtheorem{construction}[theorem]{Construction}
\newtheorem{notation}[theorem]{Notation}
\newtheorem{variant}[theorem]{Variant}

\usepackage{relsize}
\usepackage[bbgreekl]{mathbbol}
\usepackage{amsfonts}
\DeclareSymbolFontAlphabet{\mathbb}{AMSb} 
\DeclareSymbolFontAlphabet{\mathbbl}{bbold}
\newcommand{\Prism}{{\mathlarger{\mathbbl{\Delta}}}}

\usepackage{tikz}
\usepackage{tikz-cd}
\usepackage{bbm}

\title{Derived $\delta$-Rings and Relative Prismatic Cohomology}
\author{Adam Holeman}
\date{}							

\begin{document}
\maketitle

\begin{abstract} 
We characterize the relative prismatic cohomology of Bhatt and Scholze by a universal property by endowing it with the additional structure of a ``derived $\delta$-ring". This involves introducing an analogue of prismatic envelopes in the setting of filtered derived commutative rings and applying this construction to the Hodge filtration on infinitesimal cohomology. In particular, we recover relative prismatic cohomology from infinitesimal cohomology via a purely algebraic process.
\end{abstract}

\tableofcontents
\newpage


\section{Introduction}

Fix a prime $p$. To any $p$-adic formal scheme $X$, one can attach a plethora of cohomology theories, each of which captures a combination of geometric and arithmetic information present in $X$. One of the basic aims of $p$-adic Hodge theory is to compare these cohomology theories, and in doing so, explicate precisely what information is retained by each such theory. In \cite{Bhatt-Scholze}, Bhatt and Scholze introduced the theory of \emph{relative prismatic cohomology}, and established ample evidence that this theory occupies a privileged position within the landscape of such theories: most other known integral $p$-adic cohomology theories can be recovered from relative prismatic cohomology via a well-defined specialization procedure. 

Let $(A,I)$ be a prism, and denote by $\overline{A} := A/I$. For any $\overline{A}$-algebra $R$, the relative prismatic cohomology $\Prism_{R/A}$ is a commutative algebra object in the derived $\infty$-category of $A$-modules (which we will denote by $\Mod_{A}$), but in certain circumstances this object admits additional structures which universally characterize it. Namely,

\begin{proposition} \label{uni} (Proposition 7.10 in \cite{Bhatt-Scholze}) Suppose $(A,I)$ is a perfect prism and $R$ is a quasi-regular semi-perfectoid $\overline{A}$-algebra. Then $\Prism_{R/A}$ is discrete and naturally admits the structure of a prism over $(A,I)$. Moreover, it is the final object in the relative (or absolute) prismatic site of $R$.
\end{proposition}

This paper generalizes Proposition \ref{uni} to arbitrary animated commutative rings $R$ and arbitrary prisms $(A,I)$, yielding a universal characterization of relative prismatic cohomology in full generality. The universal property makes use of the notion of \emph{derived commutative rings} originally due to Akhil Mathew, and systematically studied by Arpon Raksit in \cite{Raksit}. For the remainder of this introduction, we will only discuss these objects informally, referring to the body of the paper for more details.

\subsection{\v{C}ech--Alexander Complexes}

In the case that $R$ is a smooth $\overline{A}$-algebra, relative prismatic cohomology is defined as the derived global sections of the structure sheaf on the relative prismatic site $(R/A)_{\Prism}$. A useful tool that features prominently in the study of $\Prism_{R/A}$ is the notion of a \v{C}ech--Alexander complex. This is a special instance of the generality that for any ringed site $(\mathcal{C}, \mathcal{O})$, given a cover of the final object $\mathcal{F} \xrightarrow{f} \star$ in the associated topos $Shv(\mathcal{C})$, one can compute the derived global sections of $\mathcal{O}$ via the \v{C}ech complex of $f$:
	\[ R\Gamma(\mathcal{C}, \mathcal{O}) \simeq lim_{\Delta} Hom_{Shv(\mathcal{C})}(\text{\v{C}}ech(f)^{\star}, \mathcal{O}). \]
	
The key insight in the case of prismatic cohomology is that there is a very simple procedure for producing a cover of the final object.

\begin{enumerate}
\item Choose a surjection $P \to R$ where $P$ is a polynomial $A$-algebra. Denote by $J$ the kernel of this surjection.
\item Let $F$ denote the free $\delta$-$A$-algebra on $P$. 
\item The pair $(F, J\cdot F)$ then receives a map of $\delta$-pairs from $(A,I)$, and so one can form the prismatic envelope $F\{\frac{J\cdot F}{I}\}$.
\end{enumerate}

The resulting object $F\{\frac{J \cdot F}{I}\}$ naturally resides within the prismatic site $(R/A)_{\Prism}$, and the sheaf represented by this object is a cover of the final object in the associated topos. In this way, one obtains \v{C}ech-theoretic access to relative prismatic cohomology.

Our strategy for studying $\Prism_{R/A}$ is to imitate the above framework, but with one key difference: we will never make a choice of resolution. Let us begin by explaining how the first step can be recast in an entirely choice-free way.

 The pair $J \to P$ arising in the first step of the construction can be rewritten in terms of the (derived) infinitesimal cohomology of $R$ with respect to $P$ together with its Hodge filtration:
	\[ (J \to P) \simeq (F^{1}_{H} \mathbbl{\Pi}_{R/P} \to \mathbbl{\Pi}_{R/P}). \]
Hodge-filtered infinitesimal cohomology can be computed using a \v{C}ech--Alexander complex, and so the pair $J \to P$ should be viewed as a mere approximation to the pair $F^{1}_{H} \mathbbl{\Pi}_{R/A} \to \mathbbl{\Pi}_{R/A}$, the Hodge-filtered derived infinitesimal cohomology of $R$ with respect to $A$. We will take this as our choice-free avatar for the first step of the above construction.

Of course, the derived infinitesimal cohomology $\mathbbl{\Pi}_{R/A}$ is not generally a discrete commutative ring, rather it is a (typically non-connective) $\mathbb{E}_{\infty}$-ring object in the derived category $\Mod_{A}$. We are thus confronted with the task of making sense of the notion of $\delta$-rings and prismatic envelopes in the context of higher algebra. We offer no definition of $\delta$-rings at the level of general $\mathbb{E}_{\infty}$-rings, but rather work in the setting of derived commutative rings. Forthcoming work of Benjamin Antieau (\cite{Antieau}) studies derived infinitesimal cohomology from this perspective, and we will adopt this viewpoint.

Towards these ends, we introduce an $\infty$-category of derived $\delta$-$A$-algebras over any base $\delta$-ring $A$, denoted by $\dalg(\Mod_{A})$. This category generalizes the notion of animated $\delta$-rings (as in \cite{Bhatt-Lurie}) to the non-connective setting. The category $\dalg(\Mod_{A})$ can be identified with the $\infty$-category of derived commutative $A$-algebras $\calg(\Mod_{A})$ (as in \cite{Raksit}) equipped with a lift of Frobenius modulo $p$ (see Theorem \ref{frobenius lifts on delta rings}). We then prove:

\begin{theorem} The forgetful functor
	\[ \dalg(\Mod_{A}) \to \calg(\Mod_{A}) \]
admits both left and right adjoints. We will denote the left adjoint by $Free^{\delta}_{A}$.
\end{theorem}

It therefore makes mathematical sense to contemplate the free $\delta$-$A$ algebra on $\mathbbl{\Pi}_{R/A}$. This will be refined to a filtered statement (see Notation \ref{filtered free delta}) in the course of the paper, incorporating the Hodge filtration on infinitesimal cohomology.

The final remaining input is to define a satisfactory analogue of prismatic envelopes in the setting of (filtered) derived commutative rings. Importing the prism condition into the non-connective setting presents many subtleties, but in the relative setting we can sidestep most of these technicalities by declaring that the usual rigidity condition on maps of prisms continues to hold. In particular, as soon as we specify that our 'prisms' receive a map from $(A,I)$, the datum of the Cartier divisor is completely determined by $I$ itself, and we will take the $\infty$-category of $(p,I)$-complete $\delta$-$A$-algebras, $\compdalg_{A}$ as our analogue of prisms.

Recall, the prismatic envelope takes as input a $\delta$-pair $(B,J)$ receiving a map from $(A,I)$ and associates to it the universal prism under $(A,I)$ receiving a map from $(B,J)$. Viewing the ideal $J$ as defining a filtration, we will recast the pair $(B,J)$ as a $\delta$-ring object in filtered modules over $I^{\star}A$, denoted $\dalg(\fil \Mod_{I^{\star}A})$. Given the above discussion surrounding rigidity of maps of prisms, we thus arrive at our definition of prismatic envelopes:

\begin{definition} The functor
	\[ I^{\star} \colon \compdalg(\Mod_{A}) \to \dalg(\fil \Mod_{I^{\star}A}) \]
 which endows an object $B \in \compdalg_{A}$ with the $I$-adic filtration admits a left adjoint, $\env_{I}$, which we refer to as the derived $I$-adic envelope.
\end{definition} 

Our candidate construction for prismatic cohomology then may be described as
	\[ L \Prism_{R/A} := \env_{I}(Free^{\delta}_{A}(F^{\star}_{H} \mathbbl{\Pi}_{R/A})) \]
which is a precise choice-free incarnation of the \v{C}ech--Alexander approach to computing relative prismatic cohomology. Using this construction, we prove the following generalization of Proposition 1.0.1:

\begin{theorem} The functor
	\[ L \Prism_{-/A} \colon \compcalg(\Mod_{\overline{A}}) \to \compdalg(\Mod_{A}) \]
described above is left adjoint to $- \otimes_{A} \overline{A}$. Moreover, for any $p$-complete animated commutative $\overline{A}$-algebra $R$, there is a canonical isomorphism
	\[ L\Prism_{R/A} \to \Prism_{R/A} \]
where $\Prism_{R/A}$ is the derived prismatic cohomology of \cite{Bhatt-Scholze}.

\end{theorem}

As an application this result, we give a completely formal proof of affineness of the relative prismatization, as established in Section 7.3 of \cite{Bhatt-Lurie2}.

\begin{lemma} For any $p$-complete animated commutative $\overline{A}$-algebra $R$, there is a canonical equivalence
	\[ WCart_{Spf(R)/A} \simeq Spf(L\Prism_{R/A}). \]
\end{lemma}

\subsection{Outline}
Our main task in Section 2 is to construct a satisfactory theory of derived $\delta$-rings. In Section 2.1, we review the necessary background on derived commutative rings and Goodwillie calculus following \cite{Raksit} and \cite{Brantner-Mathew}. In Section 2.2, we investigate a non-linear enhancement of the results of Section 2.1 which apply to problem of extending functors of polynomial rings rather than functors of finitely generated vector spaces. This section is somewhat lengthy, so we refer the reader to Construction \ref{non-linear right-left extension} for the main takeaway. In Section 2.3, we apply the results of the preceding sections to derived the free $\delta$-ring monad, thereby introducing the notion of a derived $\delta$-ring (see Definition \ref{derived delta rings}). In Section 2.4, we extend several classical results about $\delta$-rings to the derived setting, and study filtrations and completions of derived $\delta$-rings.

Our main task in Section 3 is to give a conceptual construction of relative prismatic cohomology utilizing the theory of derived $\delta$-rings. In Section 3.1, we review the theory of Hodge-filtered derived infinitesimal cohomology from the perspective of derived commutative rings. The Gauss-Manin connection on infinitesimal cohomology is studied, from which we construct the vertical filtration (see Theorem \ref{vertical filtration}) which plays an important technical role later on. In Section 3.2, we introduce the notion of $I$-adic envelopes and relate them to prismatic envelopes in the cases of a regular sequence (see Corollary \ref{envelopes of regular sequences}). In Section 3.3, we apply the tools of the preceding sections to give a universal construction of relative prismatic cohomology, and compare it to the site-theoretic theory of \cite{Bhatt-Scholze}.

\subsection{Acknowledgements}

It is a pleasure to thank Ben Antieau, Deven Manam, Kirill Magidson, and Noah Riggenbach for many enlightening conversations on this work and relevant related topics. Noah Riggenbach and Ben Antieau also offered extensive feedback on early drafts of the paper which significantly improved both the mathematical content and exposition. I am also grateful for the hospitality of Northwestern University, where the author was supported by the National Science Foundation under Grant No. DMS-2102010 during the Winter and Spring quarters of 2021-2022.

\subsection{Conventions}

Throughout the paper, we use the language of higher category theory and higher algebra as developed in \cite{LurieHA}, \cite{Lurie}, and \cite{kerodon}. All statements and constructions should be interepreted in a homotopy-invariant sense. For a commutative ring $A$, we will denote by $\Mod_{A}$ the $\infty$-category of $A$-module spectra, and by $CAlg(\Mod_{A})$ the $\infty$-category of $\mathbb{E}_{\infty}$-ring objects in $\Mod_{A}$. If $A$ is discrete and we wish to reference the 1-category of $A$-modules or algebras, we will decorate our categories $\Mod_{A}^{\heartsuit}$ or $CAlg(\Mod_{A})^{\heartsuit}$ with the $\heartsuit$ symbol to make this explicit.


\section{Derived $\delta$-Rings}

Fix a prism $(A,I)$ and an $\overline{A} := A/I$ algebra $R$. The relative prismatic cohomology $\Prism_{R/A}$ enjoys not only the structure of an $\mathbb{E}_{\infty}$-ring in $\Mod_{A}$, but also comes equipped with a (relative) Frobenius endomorphism $\varphi \colon \Prism_{R/A} \otimes_{A, \varphi_{A}} A \to \Prism_{R/A}$. In the case that $R$ is smooth over $\overline{A}$, this extra piece of structure arises from the fact that the structure sheaf on the relative prismatic site $(R/A)_{\Prism}$ takes values in $\delta$-rings. Since $\delta$-rings themselves come equipped with a notion of Frobenius, this structure assembles over the derived global sections. If $R$ is not smooth, one can consider the left Kan extension.

When working in positive or mixed characteristic, there are two predominant paradigms in which to express commutative multiplicative structures: $\mathbb{E}_{\infty}$-rings and animated (i.e. simplicial) commutative rings. Each of these theories supports its own generalization of the Frobenius endomorphism, but neither of these generalizations are directly related to the endomorphism $\varphi$ of prismatic cohomology. The setting of \emph{derived commutative rings} provides an intermediary between these two paradigms, and supports a natural extension of the animated Frobenius into the non-connective setting. It is in this context in which the endomorphism $\varphi$ is most naturally expressed.

Our goal in this section is to introduce the notion of a $\delta$-ring object in derived commutative rings. We will begin in Section 2.1 by reviewing some of the basic theory of derived commutative rings, referring to \cite{Raksit} and \cite{Brantner-Mathew} for many proofs. The basic theme of this section is roughly to provide a suite of techniques for extending functors defined on the heart of a $t$-structure to the entire stable $\infty$-category in question. In Section 2.2, we turn our attention towards a generalization of the techniques encountered in Section 2.1 fit to handle extension of functors defined on algebra objects in the heart of the $t$-structure. Such techniques yield access to non-connective generalizations of truncated Witt vectors, as well as Frobenius endomorphisms. In Section 2.3, we will then apply these techniques to construct an $\infty$-category of derived $\delta$-rings. These will be seen to be equivalent to derived commutative rings equipped with a lift of the Frobenius modulo $p$, generalizing the analogous statement for animated $\delta$-rings.

\subsection{Recollections on Derived Commutative Rings}

Throughout this paper we will need to work extensively with functors $\Mod_{A} \to \mathcal{D}$ where $\Mod_{A}$ is the derived category of a $\delta$-ring, and $\mathcal{D}$ is some presentable $\infty$-category (often with additional structures). Naming such functors and the concomitant coherences can be rather unwieldy. In the favorable situations encountered in this work, we will often be able to begin by constructing a functor
	\[ \Mod_{A}^{fpp} \to \mathcal{D} \] 
from the category of finitely presented projective $A$-modules, and then we will take an assortment of Kan extensions. The main example of this procedure we will encounter is the construction of the $LSym^{\delta}$-monad, where the reduction is made transparent by adopting the framework of \emph{derived algebraic contexts}, first introduced in \cite{Raksit}. The purpose of this subsection is to review the basic facts from \cite{Raksit} and fix notation which will be in use throughout the paper.

\begin{definition} A \emph{derived algebraic context} consists of a presentable symmetric monoidal stable $\infty$-category $\mathcal{C}$, a $t$-structure $(\mathcal{C}_{\geq 0}, \mathcal{C}_{\leq 0})$ compatible with the symmetric monoidal structure, and a small full subcategory $\mathcal{C}^{0} \subset \mathcal{C}^{\heartsuit}$ satisfying
\begin{itemize}
\item The $t$-structure is right complete.
\item The subcategory $\mathcal{C}^{0}$ is a symmetric monoidal subcategory, closed under symmetric powers in $\mathcal{C}^{\heartsuit}$.
\item The subcategory $\mathcal{C}^{0}$ is closed under finite coproducts and $\mathcal{P}_{\Sigma}(\mathcal{C}^{0}) \simeq \mathcal{C}_{\geq 0}$.
\end{itemize}

Compatibility of the $t$-structure with the symmetric monoidal structure means $\mathcal{C}_{\leq0}$ is closed under filtered colimits, the unit object is connective, and the tensor product of connective objects is again connective.

\end{definition}

For the most part, we will be primarily interested in the following two examples.

\begin{example} Let $k$ be a commutative ring. The derived $\infty$-category $\mathcal{C} := \Mod_{k}$ equipped with its usual symmetric monoidal structure and $t$-structure, along with the full subcategory $\mathcal{C}^{0}:= \Mod_{k}^{fpp}$ of finitely presented projective $k$-modules is a derived algebraic context.
\end{example}

The next example of interest takes place on the $\infty$-category of filtered complexes, whose definition we briefly recall. View $\Z_{\geq 0}$ as a poset with the usual ordering and endow it with the monoid structure arising from addition.
Given a commutative ring $k$, we may endow the functor category
	\[ \fil \Mod_{k} := Fun(\Z^{op}_{\geq 0}, \Mod_{k}) \]
with the Day convolution symmetric monoidal structure.
Observe that the natural evaluation functors
	\[ ev_{i} \colon \fil \Mod_{k} \to \Mod_{k}, \quad i \in \Z_{\geq 0} \]
admit fully faithful left adjoint \emph{insertion} functors
	\[ ins^{i} \colon \Mod_{k} \to \fil \Mod_{k}. \]
Concretely, these functors may be understood as follows: 

\[ ins^{i}(M)^{j} = \begin{cases}
				M & j \leq i, \\
				0 & j > i \\
				\end{cases} \]
with structure maps the identity in weights less than or equal to $i$.

\begin{example} \label{0,1 filtered algebras} We may endow $\fil \Mod_{k}$ with the neutral $t$-structure defined by
	\[ \fil \Mod_{k, \leq 0} = \{ M^{\cdot} \in \fil \Mod_{k} \big| M^{i} \in \Mod_{k, \leq 0}\} \]
	\[ \fil \Mod_{k, \geq 0} = \{ M^{\cdot} \in \fil \Mod_{k} \big| M^{i} \in \Mod_{k, \geq 0}\} \]
Equipped with this $t$-structure, $\fil \Mod_{k}$ enjoys the structure of a derived algebraic context.
The compact projective generators in this example are insertions of compact projective $k$-modules.

\end{example}

\begin{variant} In the above examples, we could replace $\Z_{\geq 0}$ with $\Delta^{1, op}$ (the opposite of the 1-simplex).
Endowing $\Delta^{1, op}$ with the symmetric monoidal structure $\star$ given by
	\[ 0 \star 0 = 0, \quad 0 \star 1 = 1 \star 0 = 1 \star 1 = 1. \]
Given a commutative ring $k$, we may endow the functor category
	\[ \pair \Mod_{k} := Fun(\Delta^{1, op}, \Mod_{k}) \]
with the Day convolution symmetric monoidal structure.
Mimicking the above construction yields a derived algebraic context on $\pair \Mod_{k}$.

\end{variant}

\begin{variant} Replacing $\Z_{\geq 0}$ with $\Z_{\geq 0}^{op}$ in the definition of $\fil \Mod_{k}$ yields \emph{increasing} filtrations, which we denote by $F_{\geq 0} \Mod_{k}$. We may replicate the above definition in this setting as well.

\end{variant}

\begin{variant} Replacing $\Z_{\geq 0}$ with $\Z_{\geq 0}^{disc}$ (the underlying discrete category) in the definition of $\fil \Mod_{k}$ yields the $\infty$-category of \emph{graded objects} in $\Mod_{k}$, denoted by $\gr \Mod_{k}$. We may replicate the above definition in this setting as well. To avoid confusion, we will distinguish graded insertions from filtered insertions notationally by
	\[ (n): \Mod_{k} \to \gr \Mod_{k} \]
and recall their explicit description:
\[ M(n)^{i} = \begin{cases}
			M & i = n, \\
			0 & else.
		\end{cases} \]
As above, the compact projective objects in the heart are precisely the insertions of the compact projective $k$-modules.
\end{variant}

We now turn our attention towards the construction of a class of algebra objects in general derived algebraic contexts. These algebra objects are built by extending the symmetric algebra monad on the heart to the entire category, a process which makes use of Goodwillie calculus.

\begin{notation} \label{sifted colimit preserving functors} For $\infty$-categories $\mathcal{C}, \mathcal{D}$, we denote by $Fun_{\Sigma}(\mathcal{C}, \mathcal{D})$ the full subcategory of $Fun(\mathcal{C}, \mathcal{D})$ spanned by those functors which preserve sifted colimits.

\end{notation}

\begin{definition}[Definition 3.21 in \cite{Brantner-Mathew}] \label{right extendable functors} Let $(\mathcal{C}, \mathcal{C}_{\leq 0}, \mathcal{C}_{\geq 0})$ be a derived algebraic context, and denote by $\Perf_{\mathcal{C}, \geq 0}$ the subcategory of compact coconnective objects in $\mathcal{C}$. Let $\mathcal{D}$ be an $\infty$-category with sifted colimits and all limits. We say a functor
	\[ F \colon \mathcal{C}^{0} \to \mathcal{D} \]
is \emph{right extendable} if the right Kan extension 
	\[ F^{R} \colon \Perf_{\mathcal{C}, \leq 0} \to \mathcal{D} \]
preserves finite coconnective geometric realizations - i.e. those finite geometric realizations of objects in $\Perf_{\mathcal{C}, \leq 0}$ whose colimit in $\mathcal{C}$ remains coconnective. 
\end{definition}

The utility of right-extendable functors is due to the following Theorem, which allows us to extend such functors to functors on the entire derived category of a ring.

\begin{theorem}[Proposition 3.14 in \cite{Brantner-Mathew}] Let $k$ be a commutative ring, and $\mathcal{D}$ an $\infty$-category with sifted colimits. Denote by $Fun_{\sigma}(\Perf_{k, \leq 0}, \mathcal{D})$ the full subcategory of $Fun(\Perf_{k, \leq 0}, \mathcal{D})$ spanned by those functors which preserve finite coconnective geometric realizations. Then the restriction functor
	\[ Fun_{\Sigma}(\Mod_{k}, \mathcal{D}) \to Fun_{\sigma}(Perf_{k, \leq 0}, \mathcal{D}) \]
is an equivalence, with inverse given by left Kan extension. The source of this functor is as in Notation \ref{sifted colimit preserving functors}.
\end{theorem}

\begin{construction} \label{linear right-left extension} Let $\mathcal{D}$ be an $\infty$-category with all limits and colimits. Denote by $Fun^{ext}(\Mod_{k}^{fpp}, \mathcal{D})$ the full subcategory of $Fun(\Mod_{k}^{fpp}, \mathcal{D})$ spanned by the right extendable functors. Then composing right Kan extension along $\Mod_{k}^{fpp} \to \Perf_{k, \leq 0}$ with left Kan extension along $\Perf_{k, \leq 0} \to \Mod_{k}$ yields the \emph{right-left extension functor}
	\[ Fun^{ext}(\Mod_{k}^{fpp}, \mathcal{D}) \xrightarrow{(-)^{RL}} Fun_{\Sigma}(\Mod_{k}, \mathcal{D}). \]
We will revisit this construction and generalize it in the next section.
\end{construction}

Our main tool for determining when a functor is right-extendable, and thus amenable to the previous construction, is through the notion of excisively polynomial functors which are always right extendable.

\begin{definition} For $n \geq 0$, we denote by $\mathbb{P}_{n}$ the power set of $\{0,..., n\}$, and for $m \leq n$, $\mathbb{P}_{n}^{\leq m}$ (respectively $\mathbb{P}_{n}^{\geq m}$) the subset of $\mathbb{P}_{n}$ consisting of those subsets of cardinality less than (greater than) $m$. Given an $\infty$-category $\mathcal{C}$, an \emph{$n$-cube} in $\mathcal{C}$ is a diagram
	\[ \chi: \mathbb{P}_{n} \to \mathcal{C}. \]
We say an $n$-cube $\chi$ is 
\begin{itemize}
\item \emph{cocartesian} if it is a colimit diagram.
\item \emph{cartesian} if it is a limit diagram.
\item \emph{strongly cocartesian} if it is left Kan extended from its restriction to $\mathbb{P}_{n}^{\leq 1}$.
\end{itemize}
Let $\mathcal{D}$ be a stable $\infty$-category. A functor $F: \mathcal{C} \to \mathcal{D}$ is \emph{$n$-excisive} if it carries strongly cocartesian $(n+1)$ cubes to cartesian cubes (which in the context of stable $\infty$-categories agree with cocartesian cubes).

\end{definition}

\begin{definition} Let $\mathcal{C}$ be a cocomplete $\infty$-category and $\mathcal{D}$ an idempotent complete additive $\infty$-category. A functor
	\[ F: \mathcal{C} \to \mathcal{D} \]
is said to be \emph{additively polynomial of degree $0$} if it is constant. Inductively, we say that $F$ is \emph{additively polynomial of degree $n$} if for every object $X$ of $\mathcal{C}$, the derivative functor
	\[ D_{X}F := fib( F(X \oplus -) \to F(-)) \]
is additively polynomial of degree $n-1$. If we do not wish to specify the degree, we will speak of $F$ as being \emph{additively polynomial}. We say $F$ is \emph{excisively polynomial} if it is additively polynomial and preserves finite geometric realizations.
\end{definition}

\begin{theorem} Given a derived algebaic context $(\mathcal{C}, \mathcal{C}_{\leq 0}, \mathcal{C}_{\geq 0})$, and a cocomplete stable $\infty$-category $\mathcal{D}$, the restriction functor 
	\[ res: Fun_{\Sigma}^{ep}(\mathcal{C}, \mathcal{D}) \to Fun_{\Sigma}^{ep}(\mathcal{C}_{\geq 0}, \mathcal{D}) \]
is an equivalence. Here the superscript $ep$ indicates we are only considering excisively polynomial functors, and the subscript $\Sigma$ is as in Notation \ref{sifted colimit preserving functors}. Furthermore, left Kan extension along the inclusion $\mathcal{C}^{\heartsuit} \to \mathcal{C}_{\geq 0}$ induces an equivalence
	\[ Fun^{ap}_{\Sigma}(\mathcal{C}^{\heartsuit}, \mathcal{D}) \xrightarrow{\simeq} Fun^{ep}_{\Sigma}(\mathcal{C}_{\geq 0}, \mathcal{D}). \]
The superscript $ap$ above refers to additively polynomial functors.
\end{theorem}

\begin{proof} See Proposition 4.2.15 in \cite{Raksit}, and Proposition 3.35 in \cite{Brantner-Mathew}.
\end{proof}

\begin{theorem} Fix a derived algebraic context $\mathcal{C}$. Then the functors
	\[ End^{ep}_{\Sigma}(\mathcal{C}_{\geq 0}) \xrightarrow{ F \to \tau_{\leq 0} \circ F \circ \iota} End^{ap}_{\Sigma}(\mathcal{C}^{\heartsuit}) \]
and
	\[ End^{ep}_{\Sigma}(\mathcal{C}_{\geq 0}, \mathcal{C}_{\geq 0}) \xrightarrow{ i_{\star}} Fun^{ep}_{\Sigma}(\mathcal{C}_{\geq 0}, \mathcal{C}) \simeq End^{ep}_{\Sigma}(\mathcal{C}) \]
are both monoidal left adjoints. The right adjoint to the first functor will be denoted by $L$ (which stands for the 'derived approximation' following \cite{Antieau}).

\end{theorem}

\begin{proof} See Proposition 2.17 and Corollary 2.20 in \cite{Antieau}.
\end{proof}

\begin{definition}[Definition 4.1.2 in \cite{Raksit}] \label{filtered monad} Let $\mathcal{C}$ be an $\infty$-category. A \emph{filtered monad} on $\mathcal{C}$ is a lax monoidal functor
	\[ \Z_{\geq 0}^{\times} \to End(\mathcal{C}) \]
where $\Z_{\geq 0}^{\times}$ is viewed as a monoidal category via multiplication. Given a full monoidal subcategory $\mathcal{E} \subset End(\mathcal{C})$, we define a \emph{filtered } $\mathcal{E}$-\emph{monad} to be a lax monoidal functor
	\[ \Z_{\geq 0}^{\times} \to \mathcal{E}. \]

\end{definition}

\begin{theorem} Let $\mathcal{C}$ be a cocomplete $\infty$-category and $F$ a filtered $\mathcal{E}$-monad on $\mathcal{C}$, where $\mathcal{E}$ is a subcategory such that every $f \in \mathcal{E}$ commutes with sequential colimits. Then $colim_{\Z_{\geq 0}}F$ is a monad on $\mathcal{C}$. 
\end{theorem}

\begin{proof} This is Proposition 4.1.4 in \cite{Raksit}.
\end{proof}

\begin{observation} Given a full monoidal subcategory $\mathcal{E} \subset End(\mathcal{C})$, denote by $F_{\geq 0} Alg(\mathcal{E})$ the $\infty$-category of filtered $\mathcal{E}$-monads (see Definition \ref{filtered monad}). The composite 
	\[ End^{ap}_{\Sigma}(\mathcal{C}^{\heartsuit}) \xrightarrow{L} End^{ep}_{\Sigma}(\mathcal{C}_{\geq 0}) \to  End^{ep}_{\Sigma}(\mathcal{C}) \]
is lax monoidal, and thus induces a functor
	\[ F_{\geq 0} Alg(End^{ap}(\mathcal{C}^{\heartsuit})) \xrightarrow{L} F_{\geq0 } Alg(End^{ep}_{\Sigma}(\mathcal{C})) \]
on filtered monad objects therein.

\end{observation}

\begin{example} Fix a derived algebraic context $(\mathcal{C}, \mathcal{C}_{\leq 0}, \mathcal{C}_{\geq 0})$, and consider the symmetric algebra functor on the heart $Sym_{\mathcal{C}}^{\heartsuit}$. We will choose to view this not as an endomorphism of $\mathcal{C}^{\heartsuit}$, but rather as a functor
	\[ \mathcal{C}^{\heartsuit} \to \mathcal{C} \]
by postcomposing with the inclusion of the heart into $\mathcal{C}$. This functor is not additively polynomial, but it has a filtration by additively polynomial subfunctors
	\[ Sym_{\mathcal{C}}^{\heartsuit, \leq n} \in Fun^{ap}_{\Sigma}(\mathcal{C}^{\heartsuit}, \mathcal{C}). \]
	
These filtered pieces organize into a filtered monad structure refining the monad structure on $Sym_{\mathcal{C}}^{\heartsuit}$, and applying the preceding observation yields a filtered monad on $\mathcal{C}$, denoted by $LSym_{\mathcal{C}}^{\leq \star}$. The colimit of this filtered monad is denoted by $LSym_{\mathcal{C}}$, and referred to as the \emph{derived symmetric algebra monad}. The $\infty$-category of algebras over this monad is denoted by $\calg(\mathcal{C})$, and objects therein are referred to as \emph{derived commutative rings} in $\mathcal{C}$.
\end{example}

\begin{lemma} \label{commuting limits and colimits in dacs} Let $\mathcal{C}$ be a stable infinity category which admits sifted colimits, and let $T$ be a sifted colimit preserving monad thereon. Then sifted colimits in $\Mod_{T}(\mathcal{C})$ commute with finite limits.  In particular, this applies to $DAlg(\mathcal{C})$ for any derived algebraic context $\mathcal{C}$.
\end{lemma}

\begin{proof} It suffices to check that finite limits commute with geometric realizations, and since any simplicial object can be written as a filtered colimit of $n$-skeletal simplicial objects, it suffices to check that finite limits commute with finite geometric realizations. By construction, the forgetful functor
	\[ F \colon \Mod_{T}(\mathcal{C}) \to \mathcal{C} \]
reflects sifted colimits and limits. Stability of $\mathcal{C}$ implies that finite limits commute with finite colimits in $\mathcal{C}$, from which we conclude.
\end{proof}

\begin{lemma} \label{techy extension result} Let $\mathcal{A}$ be an $\infty$-category which admits finite colimits and a final object, and let $\mathcal{D}$ be a cocomplete stable $\infty$-category. Suppose $F: \mathcal{A} \to \mathcal{D}$ is an $n$-excisive functor which preserves filtered colimits. Then $F$ also preserves finite totalizations and all geometric realizations.

\end{lemma}

\begin{proof} The proof of Proposition 3.37 in \cite{Brantner-Mathew} carries over verbatim to this more general setting. We recall the proof here.

Recall (from Theorem 1.8 in \cite{Goodwillie}) the inclusion of the $n$-excisive functors $Fun^{ep, n}(\mathcal{A}, \mathcal{D}) \to Fun(\mathcal{A}, \mathcal{D})$ admits a left adjoint $P_{n}$, and as $n$ varies, these adjoints organize into a tower
	\[ ...\to P_{n+1} \to P_{n} \to P_{n-1} \to ... \]
We will denote by $D_{n}(F)$ the fiber of $P_{n}(F) \to P_{n-1}(F)$, and recall that this is an $n$-homogeneous functor. The import of these recollections for us is two-fold:

\begin{itemize}
\item We can prove the claim inductively, and thus reduce to verifying the conclusion for $D_{n}(F)$.
\item Proposition 6.1.4.14 in \cite{LurieHA} posits the existence of a symmetric functor $G: (\mathcal{A}^{n} \times E\Sigma_{n})/\Sigma_{n} \to \mathcal{D}$ which preserves colimits in each variable such that
	\[ D_{n}(F) \simeq G(X,..., X)_{h \Sigma_{n}}. \]
\end{itemize}

Passing to homotopy coinvariants is exact, and thus it suffices to prove that the functor $\mathcal{A} \to \mathcal{D}$ given by $X \to G(X,..., X)$ preserves geometric realizations and finite totalizations. The first follows immediately from the fact that $\Delta^{op} \to (\Delta^{op})^{n}$ is left cofinal (Lemma 5.5.8.4 in \cite{Lurie}).

If $X^{\star} \colon \Delta \to \mathcal{A}$ is right Kan extended from $\Delta_{\leq m}$, then $(X^{\star},..., X^{\star}) \colon (\Delta)^{n} \to \mathcal{A}^{n}$ is also right Kan extended from $(\Delta_{\leq m})^{n}$. Appealing to exactness in each variable of $G$, $G(X^{\star},..., X^{\star})$ is also right Kan extended from $(\Delta_{\leq m})^{n}$, and thus
	\[G(Tot(X^{\cdot}), ..., Tot(X^{\cdot})) \simeq lim_{(\Delta)^{n}} G(X^{\cdot},..., X^{\cdot}) \simeq Tot(G(X^{\cdot},..., X^{\cdot})) \]
where the last equivalence follows from right cofinality of $\Delta \to (\Delta)^{n}$.
\end{proof}

\subsection{Functors of $k$-Algebras}

Recall, the truncated Witt vectors may be viewed as a functor
	\[ W_{n} : CAlg^{\heartsuit}_{\Z_{p}} \to CAlg^{\heartsuit}_{\Z_{p}} \]
By Kan extending from polynomial rings, one obtains a notion of \emph{animated} Witt vectors:
\begin{center}
\begin{tikzcd}
Poly_{\Z_{p}} \arrow[r, "W_{n}"] \arrow[d, "i"] 
	& \mathfrak{a}CAlg(\Mod_{\Z_{p}}) \\
\mathfrak{a}CAlg(\Mod_{\Z_{p}}) \arrow[ur, "LW_{n}"]
\end{tikzcd}
\end{center}
These animated Witt vectors play a useful role in the study of animated $\delta$-rings (see e.g. Appendix A in \cite{Bhatt-Lurie2}). In our investigation of derived $\delta$-rings, we would thus like a notion of truncated Witt vectors for derived commutative rings, but the techniques of the preceding section no longer apply to this context. Indeed, all the results encountered thus far have dealt with the extension of functors defined on the compact projective objects in the heart of a derived algebraic context, not algebra objects therein.
The goal of this section is to establish results which will allow us to extend functors of polynomial algebras to functors of derived algebras.  

Our strategy consists of reducing to the context of the previous section by precomposing with the free algebra functor. 
More precisely, fix a ring $k$. Given a functor
	\[ F: Poly_{k} \to \mathcal{D} \]
we obtain a new functor
	\[ G := F \circ Sym_{k}^{\heartsuit} : \Mod_{k}^{fpp} \to \mathcal{D} \]
which is now amenable to the techniques of the preceding section. Moreover, one can recover $F$ from $G$ by equipping $G$ with a natural piece of extra structure: namely that of a right $Sym^{\heartsuit}_{k}$-module, where the action map is induced by the monad multiplication on $Sym^{\heartsuit}_{k}$. Indeed, when endowed with this extra structure, we obtain an identification (see the proof of Lemma \ref{right-left extensions are extensions}) of functors
	\[ F \simeq colim_{\Delta^{op}} Bar_{\star}(G, Sym^{\heartsuit}_{k}, -). \]
More generally, given a monad $T$ on an $\infty$-category $\mathcal{C}$, the two sided Bar construction yields a functor
	\[ R\Mod_{T}(Fun(\mathcal{C}, \mathcal{D})) \xrightarrow{ G \to |Bar_{\star}(G,T,-)|} Fun(L\Mod_{T}(\mathcal{C}), \mathcal{D}) \]
where we are appealing to the natural right tensoring of $Fun(\mathcal{C}, \mathcal{D})$ over $End(\mathcal{C})$ in order to make sense of the left hand side. 
In this way, we reduce the problem of extending functors of polynomial algebras to that of understanding the preservation of right $Sym^{\heartsuit}_{k}$-modules under the extension procedures of the preceding section.


We begin by briefly reviewing the necessary notions regarding right-tensorings of categories following \cite{LurieHA}. 
\begin{definition}[See Definition 4.2.2.2 in \cite{LurieHA}] Let $\mathcal{C}$ be an $\infty$-category. A \emph{right action object} is a natural transformation
	\[ \alpha: M' \to M \]
in $Fun(N(\Delta^{op}), \mathcal{C})$ satisfying the following two properties:
\begin{itemize}
\item $M$ is a monoid object in $\mathcal{C}$ (see Definition 4.1.2.5 in \cite{LurieHA})
\item The maps $M'([n]) \to M'(\{n\})$ and $M'([n]) \to M([n])$ witness $M'([n])$ as the product
	\[ M'([n]) \simeq M'(\{n\}) \times M([n]) \simeq M'([0]) \times M([n]). \]
\end{itemize}
Following the standard abuse of terminology, we will refer to the $\infty$-category $M'([0])$ as being endowed with a right action of $M([1])$. We denote by $RMon(\mathcal{C})$ the full subcategory of $Fun(N(\Delta^{op}) \times \Delta^{1}, \mathcal{C})$ spanned by the right action objects.
\end{definition}

\begin{definition}[See Proposition 4.2.2.9 in \cite{LurieHA}] A \emph{right-tensoring} of an $\infty$-category $\mathcal{M}$ over a monoidal category $\mathcal{C}$ is the data of a right action object $M' \to M$ in $Cat_{\infty}$ equipped with equivalences
	\[ M \simeq \mathcal{C} \quad \text{ and } \quad M'([0]) \simeq \mathcal{M}. \]
\end{definition}

\begin{remark}[Variant 4.2.1.36 and Proposition 4.2.2.9 in \cite{LurieHA}] \label{operadic tensoring} There exists an $\infty$-operad $\mathcal{RM}^{\otimes}$ whose underlying $\infty$-category $\mathcal{RM}$ is the discrete simplicial set $\{\mathfrak{a}, \mathfrak{m}\}$ such that for any $\infty$-category $\mathcal{C}$ with finite products, there is a canonical equivalence
	\[Mon_{\mathcal{RM}}(\mathcal{C}) \simeq RMon(\mathcal{C}). \]
In particular, a right tensoring of an $\infty$-category $\mathcal{M}$ over a monoidal category $\mathcal{C}$ is equivalent to a co-Cartesian fibration of $\infty$-operads
	\[ \mathcal{M}^{\otimes} \to \mathcal{RM}^{\otimes} \]
equipped with equivalences
	\[ \mathcal{M}_{\mathfrak{a}} \simeq \mathcal{C} \quad \text{ and } \quad \mathcal{M}_{\mathfrak{m}} \simeq \mathcal{M}. \]
\end{remark}

For our purposes, all right-tensored (and left-tensored) categories will arise from (the obvious analogue of) the following construction.

\begin{construction} \label{simplicial monoids} Fix two quasi-categories $\mathcal{C}, \mathcal{M} \in sSet$, and suppose $\mathcal{C}$ is endowed with the structure of a simplicial monoid and $\mathcal{M}$ is equipped with a right action of $\mathcal{C}$ in the category of simplicial sets. The monoid structure on the simplicial set $\mathcal{C}$ yields a simplicial object in quasi-categories
	\[M_{\star} := \left(... \mathcal{C} \times \mathcal{C} \rightthreearrow \mathcal{C} \rightrightarrows [0]\right). \]
Unstraightening this diagram yields a monoidal $\infty$-category $\mathcal{C}^{\otimes}$ whose underlying $\infty$-category is $\mathcal{C}$.

Similarly, the right action of $\mathcal{C}$ on $\mathcal{M}$ yields a simplicial object in quasi-categories
	\[M'_{\star} := \left(... \mathcal{C} \times \mathcal{C} \times \mathcal{M} \rightthreearrow \mathcal{C} \times \mathcal{M} \rightrightarrows \mathcal{M}\right) \]
and the projection maps $\mathcal{C}^{n} \times \mathcal{M} \to \mathcal{C}^{n}$  induce a map of simplicial quasi-categories $M' \to M$. This yields a right action object in $Cat_{\infty}$ which we will denote by $\mathcal{M}_{\mathcal{C}}^{\otimes}$. The right action object $\mathcal{M}_{\mathcal{C}}^{\otimes}$ encodes a right tensoring of $\mathcal{M}$ over $\mathcal{C}$. 
\end{construction}

\begin{definition} A \emph{lax map of right tensored categories} $p,q \colon \mathcal{M}^{\otimes}, \mathcal{N}^{\otimes} \to \mathcal{RM}^{\otimes}$ is a map of $\infty$-operads over $\mathcal{RM}^{\otimes}$
	\[ \mathcal{M}^{\otimes} \to \mathcal{N}^{\otimes}. \]
A \emph{(strict) map of right tensored categories} as above is a lax map of right tensored categories which takes $p$-coCartesian edges to $p\circ q$-coCartesian edges.
\end{definition}

\begin{example} For general $\infty$-categories $\mathcal{C}$ and $\mathcal{D}$, $Fun(\mathcal{C}, \mathcal{D})$ is right tensored over $End(\mathcal{C})$, and this tensoring arises as in Construction \ref{simplicial monoids}.
\end{example}

\begin{example} \label{main example} Let $k$ be a commutative ring and $\mathcal{D}$ a presentable $\infty$-category. Let $End_{\Sigma}(\Mod_{k}^{\heartsuit})$ denote the full subcategory of $End(\Mod_{k}^{\heartsuit})$ spanned by the endomorphisms which preserve 1-sifted colimits. We denote by $\mathcal{E}$ the full subcategory of $End_{\Sigma}(\Mod_{k}^{\heartsuit})$ spanned by those endomorphisms $f$ such that the image of $f$ under the composite
	\[ End_{\Sigma}(\Mod_{k}^{\heartsuit}) \to Fun_{\Sigma}(\Mod_{k}^{\heartsuit}, \Mod_{k}) \simeq Fun(\Mod_{k}^{fpp}, \Mod_{k}) \]
is right-extendable in the sense of Definition \ref{right extendable functors}. Lemma \ref{monoidality of extendable endomorphisms} below exhibits $\mathcal{E}$ as a monoidal subcategory of $End_{\Sigma}(\Mod_{k}^{\heartsuit})$.
	
In this context, Construction \ref{simplicial monoids} endows the category of right-extendable functors $Fun_{\Sigma}^{ext}(\Mod_{k}^{\heartsuit}, \mathcal{D})$ with a right tensoring over $\mathcal{E}$, $Fun_{\sigma}(\Perf_{k, \leq 0}, \mathcal{D})$ with a right tensoring over $End_{\sigma}(\Perf_{k, \leq 0})$, and $Fun_{\Sigma}(\Mod_{k}, \mathcal{D})$ with a right tensoring over $End_{\Sigma}(\Mod_{k})$.
\end{example}

\begin{lemma} \label{monoidality of extendable endomorphisms} Let $\mathcal{E}$, $End_{\Sigma}(\Mod_{k}^{\heartsuit})$ be as in Example \ref{main example}. Then the inclusion $\mathcal{E} \to End_{\Sigma}(\Mod_{k}^{\heartsuit})$ is monoidal. 
\end{lemma}

\begin{proof} Both categories in question are 1-categories, and so it suffices to show that the composite of two extendable endomorphisms is once again extendable. Suppose $f,g \in \mathcal{E}$. We must show that the right Kan extension
	\[ (i \circ f \circ g)^{R} \colon \Perf_{k, \leq 0} \to \Mod_{k} \]
preserves finite coconnective geometric realizations. Before verifying this, we make some preliminary observations.

First observe that for any finite simplicial set $K$, $Fun(K, \Mod_{k})$ is compactly generated by $Fun(K, \Perf_{k})$. Since limits and colimits are computed pointwise, any functor $F \colon K \to \Mod_{k}$ which factors through $\Perf_{k}$ is compact. Indeed, given a directed diagram $\{G_{n}\}$, and a natural transformation $F \to colim_{n} G_{n}$, for each vertex $k \in K$ we obtain a factorization $F(k) \to G_{n_{k}}(k)$ for some $n_{k}$. Since $K$ was assumed to be finite, we can find a uniform $n_{k}$ so as to attain the desired factorization. To see that such functors generate the category, we appeal to stability to reduce to checking that any functor $F$ satisfying the property that for all $G \in Fun(K, \Perf_{k})$, $Hom(G, F) \simeq 0$ is equivalent to the zero functor. This follows immediately from the fact that evaluation at a vertex $i \in K$, viewed as a functor $Fun(K, \Mod_{k}) \xrightarrow{ev_{i}} \Mod_{k}$ admits a left adjoint, which reduces us to a pointwise calculation.

We next claim that given $X \in \Perf_{k, \leq 0}$, for any presentation of $(i \circ g)^{R}(X)$ as a filtered colimit of coconnective compact objects
	\[ (i \circ g)^{R}(X) \simeq colim_{n} Y^{n}, \]
the canonical map
	\[ colim_{n} (i \circ f)^{R}(Y^{n}) \to (i \circ f \circ g)^{R}(X) \]
is an equivalence. Indeed, we can present $X$ as a finite totalization of objects in $\Mod_{k}^{\heartsuit}$ - $X \simeq Tot(X_{\star})$, and then witness (using the above observation) the diagram $X_{\star}$ as a filtered colimit of finite cosimplicial objects in $\Perf_{\Mod_{k}^{\heartsuit}}$, which we will denote by $Z^{\star}_{\star}$. Then we have
	\[ (i \circ f \circ g)^{R}(X) \simeq Tot(i \circ f \circ g(X_{\star})) \simeq colim_{n} Tot(i \circ f \circ g(Z^{n}_{\star})) \]
and since $f$ preserves 1-sifted colimits, any presentation of $g(Z^{n}_{\star})$ as a filtered colimit of compact objects is preserved by $f$, whence the claim.

We now return to the main thread of the proof. Fix a diagram $X_{\star} \colon \Delta^{\leq m, op} \to \Perf_{k, \leq 0}$ such that $|X_{\star}|$ remains coconnective. Our task is to show that the canonical map
	\[ \big| (i \circ f \circ g)^{R}(X_{\star}) \big| \to (i \circ f \circ g)^{R}(\big|X_{\star}\big|) \]
is an equivalence. Since $g$ preserves discrete objects and $\Mod_{k, \leq 0} \subset \Mod_{k}$ is closed under limits, $(i \circ g)^{R}(X_{\star})$ is coconnective and we can write the diagram $(i \circ g)^{R}(X_{\star})$ as a filtered colimit of diagrams $Y^{n}_{\star}$ in $\Perf_{k, \leq 0}$. Appealing to the claim of the preceding paragraph, we see that
\begin{center}
\begin{align*}
	 \big| (i \circ f \circ g)^{R}(X_{\star}) \big| & \simeq colim_{n} \big| (i \circ f)^{R}(Y^{n}_{\star}) \big| \\
	 							     & \simeq colim_{n} (i \circ f)^{R}( \big|Y^{n}_{\star} \big|) \\
								     & \simeq (i \circ f \circ g)^{R}( \big| X_{\star} \big|)
\end{align*}
\end{center}
where the second equivalence follows from extendability of $i \circ f$, and the third from extendability of $i \circ g$.
\end{proof}

\begin{remark} The reason we restrict our attention to $\Mod_{k}$ as opposed to a general derived algebraic context is so that we may write any coconnective perfect object in $\Mod_{k}$ as a finite totalization of finitely presented projective $k$-modules. The above proof relies on this fact, which does not hold in general derived algebraic contexts. This may be avoided by replacing $\mathcal{E}$ with the category of filtered functors $\fil End^{ext}(\mathcal{C}^{\heartsuit})$ consisting of endomorphisms $\mathcal{C}^{\heartsuit} \to \mathcal{C}^{\heartsuit}$ with a filtration whose filtered pieces preserve $\mathcal{C}^{0}$ and are extendable. Working with filtered endomorphisms introduces mild complications in the forthcoming constructions, so we content ourselves with the simpler setting of $\Mod_{k}$ for this paper.
\end{remark}


\begin{lemma} \label{monoidal right Kan} Let $\mathcal{C}' \xrightarrow{i} \mathcal{C}$ be a full subcategory, and suppose that the right Kan extension of any functor $F \colon \mathcal{C}' \to \mathcal{C}$ exists. Then the composite
	\[ End(\mathcal{C}') \xrightarrow{ i \circ -} Fun(\mathcal{C}', \mathcal{C}) \xrightarrow{Ran_{\mathcal{C}'}^{\mathcal{C}}} End(\mathcal{C}) \]
is lax monoidal.
\end{lemma}

\begin{proof} Consider the full subcategory $End_{\mathcal{C}'}(\mathcal{C}) \xrightarrow{j} End(\mathcal{C})$ spanned by those endofunctors which preserve the subcategory $\mathcal{C}'$. The composite in question then factors through $End_{\mathcal{C}'}(\mathcal{C})$ as follows:

\begin{center}
\begin{tikzcd}
End_{\mathcal{C}'}(\mathcal{C}) \arrow[r, "j"]
	& End(\mathcal{C}) \\
End(\mathcal{C}') \arrow[r, "i \circ -"] \arrow[u, "R"]
	& Fun(\mathcal{C}', \mathcal{C}) \arrow[u, "Ran_{\mathcal{C}'}^{\mathcal{C}}"]
\end{tikzcd}
\end{center}

By inspection $End_{\mathcal{C}'}(\mathcal{C})$ is a monoidal subcategory of $End(\mathcal{C})$, and the natural restriction functor
	\[ res \colon End_{\mathcal{C}'}(\mathcal{C}) \to End(\mathcal{C}) \]
is a map of simplicial monoids, and thus canonically refines to a strictly monoidal functor.

It thus suffices to show that $R$ is a right adjoint to $res$ by Corollary 7.3.2.7 in \cite{LurieHA}. Let us construct the co-unit of the adjunction. Right Kan extension is right adjoint to restriction, and so we have a unit natural transformation
	\[ \hat{\epsilon} \colon res \circ Ran_{\mathcal{C}'}^{\mathcal{C}} \to Id_{Fun(\mathcal{C}', \mathcal{C})}. \]
Since $i \circ -$ is fully faithful, it suffices to construct a natural transformation of the form
	\[ i\epsilon \colon (i \circ res \circ R) \to (i \circ -). \]
We have identifications $i \circ res \circ R \simeq res \circ j \circ R \simeq res \circ Ran_{\mathcal{C}'}^{\mathcal{C}} \circ i$, and thus we declare
	\[ i \epsilon := \hat{\epsilon} \circ i. \]
To verify that the natural transformation $\epsilon$ is the counit of an adjunction between $R$ and $res$, we must verify that for any $f \in End(\mathcal{C}')$ and $g \in End_{\mathcal{C}'}(\mathcal{C})$, the composite
	\[ Hom_{End_{\mathcal{C}'}(\mathcal{C})}(f, R(g)) \xrightarrow{res} Hom_{End(\mathcal{C}')}(res(f), res(R(g))) \xrightarrow{\epsilon_{g} \circ - } Hom_{End(\mathcal{C}')}(res(f), g) \]
is a homotopy equivalence. This composite fits into the following commutative diagram of spaces

\begin{center}
\begin{tikzcd}
Hom_{End_{\mathcal{C}'}(\mathcal{C})}(f, R(g)) \arrow[r, "\epsilon \circ res(-)"] \arrow[d, "\simeq"]
		& Hom_{End(\mathcal{C}')}(res(f), g) \arrow[d, "\simeq"] \\
Hom_{End(\mathcal{C})}(j(f), j(R(g))) \arrow[r, "\hat{\epsilon} \circ res(-)"]
		& Hom_{Fun(\mathcal{C}', \mathcal{C})}(i(res(f)), i(g)). \\
\end{tikzcd}
\end{center}
The bottom arrow is a homotopy equivalence since $\hat{\epsilon}$ is the counit of an adjunction between $Ran_{\mathcal{C}'}^{\mathcal{C}}$ and $res$, from which we conclude.
\end{proof}

\begin{lemma} \label{structured right Kan} Let $\mathcal{C}' \xrightarrow{i} \mathcal{C}$ be as in the preceding lemma, and let $\mathcal{D}$ be an $\infty$-category such that the right Kan extension of any functor $F \in Fun(\mathcal{C}', \mathcal{D})$ along $i$ exists. Then the right Kan extension
	\[ Fun(\mathcal{C}', \mathcal{D}) \xrightarrow{Ran_{\mathcal{C}'}^{\mathcal{C}}} Fun(\mathcal{C}, \mathcal{D}) \]
refines to a lax map of right tensored categories, where we equip the domain and target with their natural tensorings over $End(\mathcal{C}')$ and $End(\mathcal{C})$ respectively. This refinement is such that the induced lax monoidal map $End(\mathcal{C}') \to End(\mathcal{C})$ is precisely the composite of the preceding lemma.
\end{lemma}

\begin{proof} First, observe that we may also endow $Fun(\mathcal{C}, \mathcal{D})$ with a right tensoring over $End_{\mathcal{C}'}(\mathcal{C})$ (see the notation in the proof of Lemma \ref{monoidal right Kan}) via Construction \ref{simplicial monoids}. Since $End_{\mathcal{C}'}(\mathcal{C}) \subset End(\mathcal{C})$ is a sub-simplicial monoid, the identity $Fun(\mathcal{C}, \mathcal{D}) \to Fun(\mathcal{C}, \mathcal{D})$ refines to a (strict) map of right tensored categories relating the two tensorings.

It thus suffices to prove that right Kan extension refines to a lax map of right tensored categories where $Fun(\mathcal{C}, \mathcal{D})$ is equipped with the tensoring over $End_{\mathcal{C}'}(\mathcal{C})$. This is an immediate application of Corollary 7.3.2.7 in \cite{LurieHA}, and the proof of Lemma \ref{monoidal right Kan}.
\end{proof}

The analogous situation for left Kan extensions is slightly more subtle, as the left adjoint to a monoidal functor is only oplax monoidal. Nevertheless, the following lemma will suffice for our purposes.

\begin{lemma} \label{monoidal left Kan} Let $\mathcal{C} \simeq Ind_{\kappa}(\mathcal{C}')$ be a compactly generated $\infty$-category, and let $i \colon \mathcal{C}' \to \mathcal{C}$ denote the inclusion. Then the composite
	\[ End(\mathcal{C}') \xrightarrow{i \circ -} Fun(\mathcal{C}', \mathcal{C}) \xrightarrow{Lan_{\mathcal{C}'}^{\mathcal{C}}} End(\mathcal{C}) \]
is strictly monoidal.
\end{lemma}

\begin{proof} As in the proof of Lemma \ref{monoidal right Kan}, we factor the composite through $End_{\mathcal{C}'}(\mathcal{C}) \xrightarrow{j} End(\mathcal{C})$. This yields a functor $L \colon End(\mathcal{C}') \to End_{\mathcal{C}'}(\mathcal{C})$, which is easily seen to be left adjoint to $res \colon End_{\mathcal{C}'}(\mathcal{C}) \to End(\mathcal{C}')$ via an argument analogous to that of Lemma \ref{monoidal right Kan}.

To establish the claim, we may reduce (by Corollary 7.3.2.12 in \cite{LurieHA}) to checking that for any finite collection of elements $f_{1},..., f_{n} \in End(\mathcal{C}')$, the canonical map
	\[ L(f_{1} \circ ... \circ f_{n}) \to L(f_{1}) \circ ... \circ L(f_{n}) \]
is an equivalence, which we can check pointwise. By hypothesis, left Kan extension induces an equivalence
	\[ Fun(\mathcal{C}', \mathcal{C}) \simeq End_{\omega}(\mathcal{C}) \]
and so in particular, each $L(f_{i})$ preserves filtered colimits. Given any $c \in \mathcal{C}$, the category $\mathcal{C}' \downarrow c$ is filtered, and thus we see
	\[ L(f_{1}) \circ ... \circ L(f_{n})(c) \simeq colim_{\mathcal{C}' \downarrow c} \left(f_{1} \circ ... \circ f_{n}(c')\right) \simeq L(f_{1} \circ ... \circ f_{n})(c) \]
as desired.
\end{proof}

\begin{lemma} \label{structured left Kan} Let $\mathcal{C} \simeq Ind_{\kappa}(\mathcal{C}')$ be as above and $\mathcal{D}$ be an $\infty$-category such that the left Kan extension of any functor $F \colon \mathcal{C}' \to \mathcal{D}$ along $i$ exists. Then left Kan extension
	\[ Fun(\mathcal{C}', \mathcal{D}) \xrightarrow{Lan_{\mathcal{C}'}^{\mathcal{C}}} Fun(\mathcal{C}, \mathcal{D}) \]
refines to a strict map of right tensored categories, where we equip the domain and target with their natural tensorings over $End(\mathcal{C}')$ and $End(\mathcal{C})$ respectively. This refinement is such that the induced monoidal map $End(\mathcal{C}') \to End(\mathcal{C})$ is precisely the composite of the preceding lemma.
\end{lemma}

\begin{proof} This is exactly the same as the proof of Lemma \ref{structured right Kan}, except that rather than appealing to Corollary 7.3.2.7 in \cite{LurieHA}, we must appeal to Corollary 7.3.2.12 in loc. cit., which has an additional hypothesis. This reduces to checking that for any $f \in Fun(\mathcal{C}', \mathcal{D})$ and any $g_{1},..., g_{n} \in End(\mathcal{C}')$, the canonical map
	\[ Lan_{\mathcal{C}'}^{\mathcal{C}}(f \circ g_{1} \circ ... \circ g_{n}) \to Lan_{\mathcal{C}'}^{\mathcal{C}}(f) \circ L(g_{1}) \circ ... \circ L(g_{n}) \]
is an equivalence, which follows by arguing pointwise as in the proof of Lemma \ref{monoidal left Kan}.
\end{proof}


\begin{proposition} Let $\mathcal{C}$ be a derived algebraic context, and denote by $\Perf_{\mathcal{C}, \leq 0}$ the subcategory of coconnective compact objects. Let $j : \Perf_{\mathcal{C}, \leq 0} \to \mathcal{C}$ denote the inclusion. Then the functor
	\[ End_{\sigma}(\Perf_{\mathcal{C}, \leq 0}) \xrightarrow{ (j \circ -)^{L}} End_{\Sigma}(\mathcal{C}) \]
is (strictly) monoidal, where the superscript $L$ denotes left Kan extension along $j$.
\end{proposition}

\begin{proof} This is an immediate application of Lemma \ref{monoidal left Kan}, noting that $End_{\sigma}(\Perf_{\mathcal{C}, \leq 0}) \subset End(\Perf_{\mathcal{C}, \leq 0})$ and $End_{\Sigma}(\mathcal{C}) \subset End(\mathcal{C})$ are both monoidal subcategories and that left Kan extension identifies these subcategories (see Proposition 3.13 in \cite{Brantner-Mathew}).
\end{proof}

Denote by $\mathcal{E} := End^{ext}_{\Sigma}(\Mod_{k}^{\heartsuit})$ the category of extendable endomorphisms, as defined in Example \ref{main example}.
Let $\mathcal{D}$ be a presentable $\infty$-category, and denote by $Fun_{\Sigma}^{ext}(\Mod_{k}^{\heartsuit}, \mathcal{D})$ the full subcategory of $Fun_{\Sigma}(\Mod_{k}^{\heartsuit}, \mathcal{D})$ spanned by the right-extendable functors. Our task is to show that the right-left extension functor
	\[ Fun_{\Sigma}^{ext}(\Mod_{k}^{\heartsuit}, \mathcal{D}) \to Fun_{\Sigma}(\Mod_{k}, \mathcal{D}) \]
respects the natural tensorings of these categories over $\mathcal{E}$ and $End_{\Sigma}(\Mod_{k})$ respectively. The primary observation of this section is captured in the following Proposition.

\begin{proposition} \label{structured extensions} Let $\mathcal{D}$ a complete presentable $\infty$-category. Then the right-left extension functor of Construction \ref{linear right-left extension}
	\[ (-)^{RL}: Fun^{ext}_{\Sigma}(\Mod_{k}^{\heartsuit}, \mathcal{D}) \to Fun_{\Sigma}(\Mod_{k}, \mathcal{D}) \]
refines to a lax map of right-tensored $\infty$-categories
	\[ (-)^{RL}: \left(Fun^{ext}_{\Sigma}(\Mod_{k}^{\heartsuit}, \mathcal{D})\right)^{\otimes}_{\mathcal{E}} \to \left(Fun_{\Sigma}(\Mod_{k}, \mathcal{D})\right)^{\otimes}_{End_{\Sigma}(\mathcal{C})} \]
where we are using the notation of Construction \ref{simplicial monoids} to depict the tensorings in question.
\end{proposition}

\begin{proof} Lemmas \ref{structured right Kan} and \ref{structured left Kan} imply that the composite
	\[ Fun(\Mod_{k}^{fpp}, \mathcal{D}) \xrightarrow{(-)^{R}} Fun(\Perf_{k, \leq 0}, \mathcal{D}) \xrightarrow{(-)^{L}} Fun(\Mod_{k}, \mathcal{D}) \]
admits a canonical refinement to a lax map of right-tensored categories, and our task it to show that this refinement restricts to the subcategories in question. The inclusions $Fun^{ext}(\Mod_{k}^{fpp}, \mathcal{D}) \to Fun(\Mod_{k}^{fpp}, \mathcal{D})$ and $Fun_{\Sigma}(\Mod_{k}, \mathcal{D}) \to Fun(\Mod_{k}, \mathcal{D})$ both admit natural refinements to strict maps of right-tensored categories where the domains of these inclusions are tensored over the categories $\mathcal{E}$ and $End_{\Sigma}(\mathcal{C})$ respectively. We thus obtain a natural factorization over $\mathcal{RM}^{\otimes}$ as indicated in the diagram

\begin{center}
\begin{tikzcd}
\left(Fun^{ext}(\Mod_{k}^{fpp}, \mathcal{D})\right)^{\otimes}_{\mathcal{E}} \arrow[r, dashed, "(-)^{RL}"] \arrow[d] 
	& \left(Fun_{\Sigma}(\Mod_{k}, \mathcal{D})\right)^{\otimes}_{End_{\Sigma}(\Mod_{k})} \arrow[d] \\
\left(Fun(\Mod_{k}^{fpp}, \mathcal{D})\right)^{\otimes}_{\mathcal{E}} \arrow[r, "(-)^{RL}"]
	& \left(Fun(\Mod_{k}, \mathcal{D})\right)^{\otimes}_{End_{\Sigma}(\Mod_{k})}
\end{tikzcd}
\end{center}
and it follows formally that the dashed arrow is in fact a map of $\infty$-operads.
\end{proof}

It follows formally then that right-left extension induces a functor on right $\mathcal{E}$-module objects in $Fun^{ext}_{\Sigma}(Mod_{k}^{\heartsuit}, \mathcal{D})$ to right $End_{\Sigma}(Mod_{k})$-module objects in $Fun_{\Sigma}(Mod_{k}, \mathcal{D})$. The symmetric algebra monad on $Mod_{k}^{\heartsuit}$ is an $\mathcal{E}$-monad, and so in particular, right left extension induces a functor
	\[ R\Mod_{Sym^{\heartsuit}_{k}}(Fun^{ext}_{\Sigma}(Mod_{k}^{\heartsuit}, \mathcal{D})) \to R\Mod_{LSym_{k}}(Fun_{\Sigma}(Mod_{k}, \mathcal{D})). \]
This is our main tool for extending functors from $Poly^{fg}_{k}$ to all of $DAlg(Mod_{k})$.

\begin{definition} \label{right-extendable functor of algebras} Let $\mathcal{D}$ a presentable $\infty$-category. We will say that a functor
	\[ F \colon Poly^{fg}_{k} \to \mathcal{D} \]
is \emph{right extendable} if the composite $F \circ Sym^{\heartsuit}_{k} \colon \Mod_{k}^{fpp} \to \mathcal{D}$ is right extendable in the sense of Definition \ref{right extendable functors}. We will denote by $Fun^{ext}(Poly^{fg}_{k}, \mathcal{D})$ the full subcategory of $Fun(Poly^{fg}_{k}, \mathcal{D})$ spanned by the right extendable functors.
\end{definition}

\begin{construction}[Non-linear right-left extensions] \label{non-linear right-left extension} Let $\mathcal{D}$ a presentable $\infty$-category.  We construct a functor $Fun^{ext}(Poly^{fg}_{k}, \mathcal{D}) \simeq Fun^{ext}_{\Sigma}(DAlg(\Mod_{k})^{\heartsuit}, \mathcal{D}) \to Fun_{\Sigma}(DAlg(\Mod_{k}), \mathcal{D})$ by tracing around the following diagram:

\begin{center}
\begin{tikzcd}

Fun^{ext}_{\Sigma}(DAlg(\Mod_{k})^{\heartsuit}, \mathcal{D}) \arrow[rr, dashed, "(-)^{RL}"] \arrow[d, " - \circ Sym_{k}^{\heartsuit}"]
		&& Fun_{\Sigma}(DAlg(\Mod_{k}), \mathcal{D}) \\
		
R\Mod_{Sym^{\heartsuit}_{k}}(Fun^{ext}_{\Sigma}(\Mod_{k}^{\heartsuit}, \mathcal{D})) \arrow[rr, "(-)^{RL}"]
		&& R\Mod_{LSym_{k}}(Fun_{\Sigma}(\Mod_{k}, \mathcal{D}))  \arrow[u, "\left|\text{Bar}_{\star}(-\text{,} LSym_{k} \text{,} -)\right|"]
\end{tikzcd}
\end{center}
where $Bar_{\star}(-, LSym_{\mathcal{C}}, -)$ is the two-sided Bar construction. We refer to the dashed arrow as the \emph{non-linear right-left extension functor}.
\end{construction}

We now investigate the basic properties of this construction. Our first task is to justify the use of the term 'extension'.

\begin{lemma} \label{right-left extensions are extensions} Suppose $F \in Fun^{ext}_{\Sigma}(DAlg(\Mod_{k})^{\heartsuit}, \mathcal{D})$ is a right extendable functor. Then there is a natural equivalence of functors
	\[ F^{RL} \big|_{DAlg(\Mod_{k})^{\heartsuit}} \simeq F. \]
\end{lemma}

\begin{proof} Denote by $G = (F \circ Sym^{\heartsuit}_{k})^{RL}$, and note that $G \big|_{\Mod_{k}^{\heartsuit}} \simeq F \circ Sym^{\heartsuit}_{k}$. In particular, for any $P \in DAlg(\Mod_{k})^{\heartsuit}$, the right left extension is the geometric realization of the simplicial object
	\[ Bar_{\star}(F \circ Sym^{\heartsuit}_{k}, Sym^{\heartsuit}_{k}, P) \simeq F(Bar_{\star}(Sym^{\heartsuit}_{k}, Sym^{\heartsuit}_{k}, P)). \]
The Bar resolution of $P$ as a left $Sym^{\heartsuit}_{k}$-module admits a natural augmentation to $P$ itself, and thus by applying $F$ we obtain a natural transformation
	\[ F^{RL}\big|_{\Mod_{k}^{\heartsuit}} \to F. \]
To check it is an isomorphism, it suffices to check on the subcategory $Poly^{fg}_{k}$ since both functors preserve 1-sifted colimits. Restricted to this subcategory, the augmentation of the Bar complex admits a natural splitting given by the unit of the $Sym^{\heartsuit}_{k}$-module structure on $P$ (this splitting only exists on the level of underlying complexes since the unit is not multiplicative, but this is sufficient for our purposes), from which we conclude.
\end{proof}

\begin{observation} The non-linear right-left extension of a functor manifestly preserves sifted colimits. This property combined with the conclusion of the preceding lemma immediately implies that the restriction of $F^{RL}$ to the connective objects is precisely the left Kan extension of $F$ to animated commutative rings:
	\[ F^{RL} \big|_{DAlg(\Mod_{k})} \simeq Lan_{Poly^{fg}_{k}}^{\mathfrak{a}CAlg_{k}}(F). \]
\end{observation}

In order to understand the behavior of the non-linear right-left extension on non-connective objects, it is useful to understand what types of limits are preserved by the extension. Our main result in this direction is the following:

\begin{theorem} \label{non-linear techy extension result} Suppose $\mathcal{D}$ is a stable $\infty$-category and $T$ is a sifted colimit preserving monad thereon. If $F \in Fun(Poly^{fg}_{k}, L\Mod_{T}(\mathcal{D}))$ satisfies the property that $F \circ Sym_{k}^{\heartsuit}$ admits an exhaustive filtration by excisively polynomial subfunctors. Then $F^{RL}$ preserves finite totalizations.
\end{theorem}

\begin{proof} Fix a diagram
	\[ X_{\star} \colon \Delta^{\leq n} \to DAlg(\Mod_{k}). \]
We first observe that the existence of the hypothesized filtration on $F \circ Sym^{\heartsuit}_{k}$ guarantees that
	\[ G := (F \circ Sym^{\heartsuit}_{k})^{RL} \]
preserves finite totalizations by Lemma \ref{techy extension result}, as does $LSym_{k}$. In particular, the canonical map of simplicial objects in $L\Mod_{T}(\mathcal{D})$
	\[ Bar_{\star}(G, LSym_{k}, Tot(X_{\star})) \to Tot(Bar_{\star}(G, LSym_{k}, X_{\star})) \]
is an equivalence, since the $n$-simplices of the Bar resolution are canonically identified with $G \circ (LSym_{k})^{\circ n}(-)$. By Lemma \ref{commuting limits and colimits in dacs}, geometric realizations commute with finite limits in $L\Mod_{T}(\mathcal{D})$, and thus we obtain
\begin{center}
\begin{align*}
F^{RL}(Tot(X_{\star})) & := \big| Bar_{\star}(G, LSym_{k}, Tot(X_{\star})) \big|\\
				  &  \simeq \big| Tot(Bar_{\star}(G, LSym_{k}, X_{\star})) \big| \\
				  & \simeq Tot\left(\big| Bar_{\star}(G, LSym_{k}, X_{\star}) \big|\right) \\
				  & \simeq Tot(F^{RL}(X_{\star}))
\end{align*}
\end{center}
as desired.
\end{proof}

\begin{example} \label{derived truncated Witt vectors} Consider the $n$-truncated Witt vectors
	\[ W_{n} \colon Poly^{fg}_{\Z} \to DAlg(\Mod_{\Z}). \]
We claim that $W_{n}$ is right extendable. Since the forgetful functor $DAlg(\Mod_{\Z}) \to \Mod_{\Z}$ reflects sifted colimits and limits, it suffices to verify right-extendability of the composite
	\[ \Mod^{fpp}_{\Z} \xrightarrow{Sym^{\heartsuit}_{\Z}} Poly^{fg}_{\Z} \xrightarrow{W_{n}} DAlg(\Mod_{\Z}) \to \Mod_{\Z} \]
which we will denote by $\mathscr{W}_{n}$. The Verschiebung fits in to a fiber sequence
	\[ \mathscr{W}_{1} \xrightarrow{V^{n}} \mathscr{W}_{n} \to \mathscr{W}_{n-1} \]
in $Fun(\Mod^{fpp}_{\Z}, \Mod_{\Z})$. Right Kan extending yields a fiber sequence
	\[ \mathscr{W}_{1}^{R} \to \mathscr{W}_{n}^{R} \to \mathscr{W}_{n-1}^{R}. \]
Observe that we have a natural identification $\mathscr{W}_{1} \simeq Sym^{\heartsuit}_{\Z}$, and thus $\mathscr{W}_{1}$ is right extendable. Proceeding inductively, the above fiber sequence implies that $\mathscr{W}_{n}^{R}$ preserves finite coconnective geometric realizations, and thus $W_{n}$ is right extendable as desired.
\end{example}


\subsection{The $\dsym$-Monad}

Let $\mathfrak{a}CAlg^{\delta}(\Mod_{\Z_{p}})$ denote the $\infty$-category of animated $\delta$-rings over $\Z_{p}$ (see e.g. Appendix A of \cite{Bhatt-Lurie2}). We will denote by $\Mod_{\Z_{p}}$ the derived $\infty$-category of $\Z_{p}$.

Recall from Proposition A.20 of \cite{Bhatt-Lurie2} that the forgetful functor
	\[ \mathfrak{a}CAlg^{\delta}(\Mod_{\Z_{p}}) \to \Mod_{\Z_{p}, \geq 0} \]
admits a left adjoint, and in fact the resulting adjunction is monadic. Our goal in this section is to extend this monad to the entire derived category, using the results of Section 2.1.

To accomplish this, we may restrict the monad to free $\Z_{p}$-modules, where we will denote the monad by $Sym^{\delta}_{\Z_{p}}$. Our strategy is to introduce a filtration of $Sym^{\delta}_{\Z_{p}}$ by subfunctors $Sym^{\delta, \leq n}_{\Z_{p}}$ such that the following three conditions hold:
\begin{itemize}
\item The filtration is exhaustive.
\item Each filtered piece $Sym^{\delta, \leq n}_{\Z_{p}}$ is additively polynomial.
\item The assignment $n \to Sym^{\delta, \leq n}_{\Z_{p}}$ may be endowed with the structure of a filtered monad refining the monad structure on $Sym^{\delta}_{\Z_{p}}$.
\end{itemize}

\begin{notation} \label{free delta notation} Let $S$ be a set. We will denote by $\Z_{p}\left< S\right>$ the free $\Z_{p}$-module generated by $S$. Evaluating $Sym^{\delta}_{\Z_{p}}$ on $\Z_{p}\left<S\right>$ yields
	\[ Sym^{\delta}_{\Z_{p}}(\Z_{p}\left<S\right>) \simeq \Z_{p}\left[x_{s,i} \big| s \in S, i \in \Z_{\geq 0} \right] \]
where the right hand side is the module underlying the polynomial algebra on the depicted generators. In this notation the $\delta$-operator acts by $\delta(x_{s,i}) = x_{s, i+1}$.
\end{notation}

\begin{definition} Given a monomial
	\[ f = \lambda x_{s_{1}, i_{1}}^{j_{1}}...x_{s_{n}, i_{n}}^{j_{n}} \]
in $Sym^{\delta}_{\Z_{p}}(\Z_{p}\left<S\right>)$, we define the \emph{$\delta$-degree} of $f$ to be
	\[ deg_{\delta}(f) := \sum_{k=1}^{n} p^{i_{k}}j_{k}. \]
We extend this to general polynomials in the obvious way: $deg_{\delta}(f_{1} + f_{2}) = \text{max}\{deg_{\delta}(f_{1}), deg_{\delta}(f_{2}) \}.$

\end{definition}

\begin{definition} For a set $S$, we define
	\[ Sym^{\delta, \leq n}_{\Z_{p}}(\Z_{p}\left<S\right>) \subset Sym^{\delta}_{\Z_{p}}(\Z_{p}\left<S\right>)\]
to be the submodule generated by those polynomials whose $\delta$-degree is less than or equal to $n$.
\end{definition}

\begin{observation} Given a map of free $\Z_{p}$-modules $f\colon M \to N$, the associated map of $\delta$-rings $Sym^{\delta}_{\Z_{p}}(f)$ does not necessarily preserve $\delta$-degree (to see this, take any non-trivial element of the kernel of $f$). However, it does preserve the filtration by $\delta$-degree, since it can only decrease the $\delta$-degree. In particular, we may view the filtration as a functor
	\[ Sym^{\delta, \leq \star}_{\Z_{p}} \colon \Mod_{\Z_{p}}^{free} \to Fun(\Z_{\geq 0}^{op}, \Mod_{\Z_{p}}) = F_{\geq 0} \Mod_{\Z_{p}} \]
from free $\Z_{p}$-modules into filtered complexes (with increasing filtrations).

\end{observation}

\begin{lemma} 
The functor 
	\[ Sym_{\Z_{p}}^{\delta, \leq \star} \colon \Mod_{\Z_{p}}^{free} \to Fun(\Z_{\geq 0}^{op}, \Mod_{\Z_{p}}) \]
satisfies
	\[ Sym^{\delta, \leq n}_{\Z_{p}}(M \oplus N) \simeq \bigoplus_{i + j = n} Sym^{\delta, \leq i}_{\Z_{p}}(M) \otimes_{\Z_{p}} Sym^{\delta, \leq j}_{\Z_{p}}(N). \]

\end{lemma}

\begin{proof} This is a simple unwinding of definitions.
\end{proof}

\begin{corollary} Let $i \colon \Mod_{\Z_{p}}^{free} \to \Mod_{\Z_{p}}$ be the inclusion. Then $i \circ Sym^{\delta, \leq n}$ is additively polynomial of degree $n$ for all $n$ (where we view $Sym^{\delta, \leq n}$ as an endomorphism of $\Mod_{\Z_{p}}^{free}$).

\end{corollary}

\begin{proof} Appealing to the preceding lemma, we can express the derivative at $X$ via
\begin{align*}
D_{X}(i \circ Sym_{\Z_{p}}^{\delta, \leq n})(M) & := fib(Sym^{\delta, \leq n}_{\Z_{p}}(X \oplus M) \to Sym_{\Z_{p}}^{\delta, \leq n}(M)) \\
	& \simeq fib\left( \bigoplus_{i + j = n} (Sym_{\Z_{p}}^{\delta, \leq i}(X) \otimes_{\Z_{p}} Sym_{\Z_{p}}^{\delta, \leq j}(M)) \to Sym_{\Z_{p}}^{\delta, \leq n}(M) \right) \\
	& \simeq \bigoplus_{i + j = n, j \not = n} Sym_{\Z_{p}}^{\delta, \leq i}(X) \otimes_{\Z_{p}} Sym_{\Z_{p}}^{\delta, \leq j}(M)
\end{align*}
from which the claim now follows by induction on $n$ along with the fact that tensoring is exact.
\end{proof}

\begin{proposition} \label{delta degree} Fix an element $f \in Sym^{\delta}_{\Z_{p}}(\Z_{p}\left<S\right>)$ of $\delta$-degree $n$. Then
	\[ deg_{\delta}(\delta(f)) = pn. \]

\end{proposition}

\begin{proof} We first reduce to the case that $f$ is a monomial, which follows immediately from the $\delta$-ring relation
	\[ \delta(f_{1} + f_{2}) = \delta(f_{1}) + \delta(f_{2}) + \dfrac{f_{1}^{p} + f_{2}^{p} - (f_{1} + f_{2})^{p}}{p}. \]

We may thus take $f$ to be a monomial, in which case we will proceed by induction on $n$. Notice that the base case is immediate from the definition of $\delta$-degree. If $f = x_{s, j}$ for some $j$, the claim is obvious. Otherwise we may write
	\[ f = x_{s,j} \tilde{f} \]
for some $s \in S$ where $deg_{\delta}(\tilde{f}) \leq (n-1)$. Applying the $\delta$ operation then yields
	\[ \delta(f) = x_{s,0}^{p} \delta(\tilde{f}) + \tilde{f}^{p} \delta(x_{s,0}) + p \delta(x_{s,0}) \delta(\tilde{f}). \]
Analyzing each of the summands individually and applying the induction hypothesis reveals that each term is precisely of degree $pn$, as desired.
\end{proof}

\begin{corollary} The functor
	\[ \Z_{\geq 0} \xrightarrow{Sym_{\Z_{p}}^{\delta, \leq \star}} End(\Mod_{\Z_{p}}^{free}) \]
is lax-monoidal. Furthermore, for any $n$ and $m$, the diagram of natural transformations
\begin{center}
\begin{tikzcd}
Sym^{\delta, \leq n}_{\Z_{p}} \circ Sym^{\delta, \leq m}_{\Z_{p}} \arrow[r, "i"] \arrow[d, "\mu"] 
	& Sym^{\delta}_{\Z_{p}} \circ Sym^{\delta}_{\Z_{p}} \arrow[d, "\mu"] \\
Sym^{\delta, \leq nm}_{\Z_{p}} \arrow[r, "i"] 
	& Sym^{\delta}_{\Z_{p}}
\end{tikzcd}
\end{center}
commutes, where $\mu$ is the monad structure.

In particular, $Sym_{\Z_{p}}^{\delta, \leq \star}$ enjoys the structure of filtered monad which is furthermore compatible with the monad structure on $Sym_{\Z_{p}}^{\delta}$.

\end{corollary}

\begin{proof} Since the structure maps of the filtration $i \colon Sym^{\delta, \leq n}_{\Z_{p}} \to Sym^{\delta}_{\Z_{p}}$ are pointwise monomorphisms, it suffices to prove that the composite 
	\[ Sym^{\delta, \leq n}_{\Z_{p}} \circ Sym^{\delta, \leq m}_{\Z_{p}} \xrightarrow{i} Sym^{\delta}_{\Z_{p}} \circ Sym^{\delta}_{\Z_{p}} \xrightarrow{\mu} Sym^{\delta}_{\Z_{p}} \]
factors through $Sym^{\delta, \leq nm}$. Indeed, if such a factorization exists, it is necessarily unique, and the lax monoidality follows formally.

We can check the factorization pointwise, so fix $M = \Z_{p}\left<S\right>$. Define
	\[ \tilde{S} := \left\{\text{monic monomials of $\delta$-degree } \leq m \text{ in } Sym_{\Z_{p}}^{\delta}(M)\right\} \]
so in particular, we obtain
	\[ Sym^{\delta, \leq n} \circ Sym^{\delta, \leq m}(M) \simeq Sym^{\delta, \leq n}(\Z_{p}\left< \tilde{S}\right>). \]
	
Fix a monomial $f = \lambda f_{s_{1}, i_{1}}^{j_{1}}... f_{s_{n}, i_{n}}^{j_{n}}$ in $Sym^{\delta, \leq n}(\Z_{p}\left< \tilde{S}\right>)$ (see Notation \ref{free delta notation}). Our task is to show that the degree of $\mu \circ i(f)$ is bounded above by $nm$. Unpacking definitions, we observe that
	\[ \mu \circ i(f) =  \lambda (\delta^{i_{1}}(f_{s_{1}}))^{j_{1}}... (\delta^{i_{n}}(f_{s_{n}}))^{j_{n}} \]
whose degree is (courtesy of Proposition \ref{delta degree})  given by
	\[ deg_{\delta}(\mu \circ i(f)) = \sum_{k=1}^{n} j_{k} p^{i_{k}} deg_{\delta}(f_{s_{k}}). \]
By hypothesis, we know that $\sum_{k=1}^{n} j_{k} p^{i_{k}} \leq m$, and $deg_{\delta}(f_{s_{k}}) \leq n$ for all $k$, from which we conclude.
\end{proof}

\begin{theorem} Denote by $LSym^{\delta, \leq n}_{\Z_{p}}$ the right-left extension of $Sym_{\Z_{p}}^{\delta, \leq n}$ to $\Mod_{\Z_{p}}$. Consider the colimit
	\[ LSym_{\Z_{p}}^{\delta} := colim_{n} LSym^{\delta, \leq n} \in End_{\Sigma}(\Mod_{\Z_{p}}). \]
Then $LSym_{\Z_{p}}^{\delta}$ inherits the structure of a monad. We refer to this as the \emph{derived} $\delta$-\emph{algebra monad}.

\end{theorem}

\begin{proof} We have already equipped $LSym^{\delta, \leq \star}_{\Z_{p}}$ with the structure of a filtered monad. The theorem then follows by appealing to Proposition 4.1.4 of \cite{Raksit}.
\end{proof}

\begin{definition} \label{derived delta rings} We denote by $\dalg(\Mod_{\Z_{p}})$ the $\infty$-category of $\dsym_{\Z_{p}}$-algebras, and refer to this as the $\infty$-category of \emph{derived} $\delta$-\emph{rings} over $\Z_{p}$.

\end{definition}

\begin{observation} \label{connective delta rings} Denote by
	\[ \dalg(\Mod_{\Z_{p}})_{\geq 0} := \dalg(\Mod_{\Z_{p}}) \times_{\Mod_{\Z_{p}}} \Mod_{\Z_{p}, \geq 0}. \]
It follows formally that the forgetful functor $\dalg(\Mod_{\Z_{p}})_{\geq 0} \to \Mod_{\Z_{p}, \geq 0}$ preserves sifted colimits and admits a left adjoint. In particular, the adjunction is monadic and by Proposition A.20 in \cite{Bhatt-Lurie2}, we obtain a canonical identification
	\[ \dalg(\Mod_{\Z_{p}})_{\geq 0} \simeq \mathfrak{a}CAlg^{\delta}(Mod_{\Z_{p}}) \]
where the right hand side is the $\infty$-category of animated $\delta$-rings as introduced in Appendix A of \cite{Bhatt-Lurie2}.

\end{observation}

\begin{proposition} \label{limits and colimits} The forgetful functor
	\[ \dalg(\Mod_{\Z_{p}}) \xrightarrow{F} CAlg(\Mod_{\Z_{p}}) \]
preserves small limits and colimits.
\end{proposition}

\begin{proof} We follow the proof of Proposition 4.2.27 in \cite{Raksit} essentially verbatim.

We first observe that $F$ preserves limits and sifted colimits. Indeed, the composite
	\[ \dalg(\Mod_{\Z_{p}}) \xrightarrow{F} CAlg(\Mod_{\Z_{p}}) \to  \Mod_{\Z_{p}} \]
preserves limits and sifted colimits by construction. On the other hand, the forgetful functor
	\[ CAlg(\Mod_{\Z_{p}}) \to \Mod_{\Z_{p}} \]
is conservative, and thus $F$ must also preserves limits and sifted colimits. It remains to show that $F$ preserves binary coproducts.

Since $\dalg(\Mod_{\Z_{p}})$ is the category of algebras over the $\dsym_{\Z_{p}}$ monad, for any object $A \in \dalg(\Mod_{\Z_{p}})$ we have a canonical Bar resolution
	\[ A \simeq colim_{\Delta^{op}} (\dsym_{\Z_{p}})^{\circ (\star + 1)}(A) \]
which reduces us to verifying that $F$ preserves coproducts of free algebras.

So fix free algebras $A = \dsym_{\Z_{p}}(M)$ and $B = \dsym_{\Z_{p}}(N)$. Our task is to verify that the natural map
	\[ F(A) \otimes_{\Z_{p}} F(B) \to F(A \sqcup B) \]
is an equivalence. Since $A$ and $B$ are free, we see that
	\[ A \sqcup B \simeq \dsym_{\Z_{p}}(M \oplus N). \]
Under this identification, we can refine the comparison map of commutative algebra objects
	\[ \dsym_{\Z_{p}}(M) \otimes \dsym_{\Z_{p}}(N) \to \dsym_{\Z_{p}}(M \oplus N) \]
to a map of filtered algebras
	\[ LSym_{\Z_{p}}^{\delta, \leq \star}(M) \otimes_{\Z_{p}} LSym_{\Z_{p}}^{\delta, \leq \star}(N) \to LSym_{\Z_{p}}^{\delta, \leq \star}(M \oplus N) \]
where the tensor product is now given by Day convolution. Since the filtered pieces are excisively polynomial, we can reduce to the case where $M$ and $N$ are connective, and thus to the case that $M$ and $N$ are free $\Z_{p}$-algebras, where the claim is immediate.
\end{proof}

\begin{definition} \label{derived Witt vectors} The right adjoint to the forgetful functor $\dalg(Mod_{\Z_{p}}) \to \calg(Mod_{\Z_{p}})$ will be denoted by $W$, and referred to as the \emph{derived Witt vectors} functor. The left adjoint to the forgetful functor will be denoted by $Free^{\delta}_{\Z_{p}}$ and referred to as the \emph{free} $\delta$-\emph{ring} functor.
\end{definition}

\begin{observation} \label{animated vs derived Witt vectors} The conclusions of Proposition \ref{limits and colimits} apply equally as well to connective objects, yielding left and right adjoints to the forgetful functor
	\[ \dalg(\Mod_{\Z_{p}})_{\geq 0} \to \mathfrak{a}CAlg(Mod_{\Z_{p}}). \]
The identification
	\[ \dalg(\Mod_{\Z_{p}})_{\geq 0} \simeq \mathfrak{a}CAlg^{\delta}(Mod_{\Z_{p}}) \]
of Observation \ref{connective delta rings} then identifies the right adjoint to the above with the animated Witt vectors of \cite{Bhatt-Lurie2}.
\end{observation}


\subsection{Constructions and Examples}

In the connective setting (see Appendix A in \cite{Bhatt-Lurie2}), there is an equivalence between animated $\delta$-rings and animated rings equipped with a lift of Frobenius. This is proven by appealing to the left Kan extension of the length 2 Witt vector functor and \emph{defining} animated $\delta$-rings to be animated rings equipped with a section of the canonical projection $\pi \colon W_{2}(R) \to R$. The desired description then follows by appealing to the usual pullback square expressing the Frobenius $W_{2}(R) \to R$ as the pullback of $Frob\colon R \to R \otimes^{\mathbb{L}} \mathbb{F}_{p}$ along the reduction mod $p$ map.

An analogous description holds in the derived setting: derived $\delta$-rings are nothing other than derived commutative rings equipped with a lift of Frobenius (see Proposition \ref{frobenius lifts on delta rings} below). To prove this however, we will not follow the above outline. Rather, we will use our monadic definition of $\delta$-rings. 

For any category $\mathcal{C}$, one can concoct a new $\infty$-category consisting of objects of $\mathcal{C}$ equipped with an endomorphism by contemplating functors from the classifying space of the monoid $\N$ into $\mathcal{C}$: $Fun(B\N, \mathcal{C})$. Any functorial assignment of an endomorphism to an object of $\mathcal{C}$ can then be expressed as a section of the evaluation map $ev_{\star} \colon Fun(B\N, \mathcal{C}) \to \mathcal{C}$. 
	
\begin{construction} \label{Frobenius on derived algebras} Consider the functor
	\[ Frob \colon Poly^{fg}_{\mathbb{F}_{p}} \to Fun(B\N, Poly_{\mathbb{F}_{p}}) \xrightarrow{ i \circ -} Fun(B\N, DAlg(\Mod_{\mathbb{F}_{p}})) \]
which assigns a polynomial $\mathbb{F}_{p}$-algebra its Frobenius endomorphism. We claim that $Frob$ is right-extendable in the sense of Definition \ref{right-extendable functor of algebras}, which is to say that the right Kan extension of $Frob \circ Sym^{\heartsuit}_{\mathbb{F}_{p}}$ to perfect coconnective $\mathbb{F}_{p}$-modules preserves finite coconnective geometric realizations. The composite
	\[ Fun(B\N, DAlg(\Mod_{\mathbb{F}_{p}})) \xrightarrow{ev_{\star}} DAlg(\Mod_{\mathbb{F}_{p}}) \to \Mod_{\mathbb{F}_{p}} \]
reflects sifted colimits and limits, and thus it suffices to prove that the composite
	\[ \Mod_{\mathbb{F}_{p}}^{fpp} \xrightarrow{Sym^{\heartsuit}_{\mathbb{F}_{p}}} Poly^{fg}_{\mathbb{F}_{p}} \xrightarrow{Frob} Fun(B\N, DAlg(\Mod_{\mathbb{F}_{p}})) \to \Mod_{\mathbb{F}_{p}} \]
is right extendable, but this is precisely $Sym^{\heartsuit}_{\mathbb{F}_{p}}$, which we know to be right-extendable.

Construction \ref{non-linear right-left extension} thus yields an extension of $Frob$ to a functor of the form
	\[ DAlg(\Mod_{\mathbb{F}_{p}}) \to Fun(B\N, DAlg(\Mod_{\mathbb{F}_{p}})) \]
which we will continue to denote by $Frob$.

\end{construction}

\begin{remark} The functor $Frob$ as defined above preserves small limits and colimits. Indeed, since $B\N$ has only a single object, evaluation at this object reflects both limits and colimits
	\[ ev_{\star} \colon Fun(B\N, \calg(\Mod_{\mathbb{F}_{p}})) \to \calg(\Mod_{\mathbb{F}_{p}}). \]
By construction, $Frob$ is a section of $ev_{\star}$, and so the claim follows at once.

\end{remark}

\begin{definition} Let $\calg(\Mod_{\Z_{(p)}})^{Frob}$ denote the pullback

\begin{center}
\begin{tikzcd}
\calg(\Mod_{\Z_{(p)}})^{Frob} \arrow[r] \arrow[d] 
	& Fun(B\N, \calg(\Mod_{\Z_{(p)}})) \arrow[d, " \pi_{\star}"] \\
\calg(\Mod_{\mathbb{F}_{p}}) \arrow[r, "Frob"]
	& Fun(B\N, \calg(\Mod_{\mathbb{F}_{p}}))
\end{tikzcd}
\end{center}
where $\pi \colon \calg(\Mod_{\Z_{(p)}}) \to \calg(\Mod_{\mathbb{F}_{p}})$ is reduction modulo $p$.
\end{definition}

\begin{theorem} \label{frobenius lifts on delta rings} The $\infty$-category $\dalg$ is equivalent to the category of derived algebras equipped with a Frobenius lift.

\end{theorem}

\begin{proof} Since both $Frob$ and $\pi_{\star}$ preserve small limits and colimits, so too does the top horizontal arrow
	\[ i\colon \calg(\Mod_{\Z_{(p)}})^{Frob} \to Fun(B\N, \calg(\Mod_{\Z_{(p)}})). \]
Postcomposing with $ev_{0}\colon Fun(B\N, \calg(\Mod_{\Z_{(p)}})) \to \calg(\Mod_{\Z_{(p)}})$ which also preserves such limits and colimits, we deduce that $\calg(\Mod_{\Z_{(p)}})^{Frob} \to \Mod_{\Z_{(p)}}$ admits a left adjoint and the resulting adjunction is monadic by Barr-Beck-Lurie. Since $\calg(\Mod_{\Z_{(p)}})^{Frob} \to \calg(\Mod_{\Z_{(p)}})$ preserves small limits and colimits, the monad is left right extended from its restriction to $\Mod_{\Z_{(p)}}^{free}$, and thus it suffices to identify the monad restricted to $\Mod_{\Z_{(p)}}^{free}$ with $Sym^{\delta}$, but this is classical.
\end{proof}


It will be useful to contemplate not only $\delta$-algebras, but filtrations and completions on such objects.

\begin{variant} We define the $\infty$-category of \emph{filtered $\delta$-rings} as the fiber product 
 \[  \dalg(\fil \Mod_{\Z_{p}}) :=  \calg(\fil \Mod_{\Z_{p}}) \times_{\calg(\Mod_{\Z_{p}})} \dalg(\Mod_{\Z_{p}}).\]
 Similarly, the $\infty$-category of  \emph{$\delta$-pairs} is denoted by $ \dalg(\pair \Mod_{\Z_{p}})$, and defined analogously.
\end{variant}

\begin{lemma} \label{filtered sym} The forgetful functor $F\colon \dalg(\fil \Mod_{\Z_{p}}) \to \fil \Mod_{\Z_{p}}$ preserves small limits and sifted colimits. In particular, it admits a left adjoint, denoted by $\dsymneut_{\Z_{p}}$, and the resulting adjunction is monadic.
\end{lemma}

\begin{proof} Since the forgetful functor $ \calg(\fil \Mod_{\Z_{p}}) \to \fil \Mod_{\Z_{p}}$ is conservative and preserves limits and sifted colimits, it suffices to prove that the forgetful functor
	\[  \dalg(\fil \Mod_{\Z_{p}}) \to  \calg(\fil \Mod_{\Z_{p}}) \]
preserves limits and sifted colimits. This follows from the definition of $ \dalg(\fil \Mod_{\Z_{p}})$ as a fiber product since all the $\infty$-categories in question are presentable, and the functors in the defining diagram preserve both limits and colimits (see Proposition \ref{limits and colimits}, as well as Propositions 5.5.3.13 and 5.5.3.18 in \cite{Lurie}).
\end{proof}

\begin{notation} \label{filtered free delta} The proof of Lemma \ref{filtered sym} establishes the existence of a left adjoint to the forgetful functor
	\[ \dalg(\fil \Mod_{\Z_{p}}) \to \calg(\fil \Mod_{\Z_{p}}) \]
which we will denote by $Free^{\delta, [0]}_{\Z_{p}}$.

\end{notation}

\begin{definition} \label{delta A algebras} Let $A$ be a (possibly derived) $\delta$-ring. We define the $\infty$-category of derived $\delta$-$A$-algebras to be the slice category
	\[ \dalg(\Mod_{A}) := \dalg(\Mod_{\Z_{p}})_{A \downarrow}. \]
\end{definition}

Let $(A,I) \in \dalg(F^{\{0,1\}}\Mod_{\Z_{p}})$. We now wish to define $I$-complete $\delta$-rings. In order to phrase this in the level of generality that we need, we introduce some notation.

\begin{notation} \label{I-adic filtration}
Let $p_{!}\colon \calg(\pair \Mod_{A}) \to \calg(\fil \Mod_{A})$ denote the left adjoint to restriction. We will denote by $I^{\star}A$ the filtered algebra given by $p_{!}(A,I)$ and refer to this as the $I$-adic filtration on $A$.
\end{notation}

\begin{definition} We say a map of $A$-modules $M \to N$ is an $A/I$-equivalence if $M \otimes_{A} A/I \to N \otimes_{A} A/I$ is an equivalence. Let $S$ denote the class of $A/I$-equivalences. We define the subcategory of \emph{$I$-adically complete} $A$-modules to be the full subcategory of $\Mod_{A}$ spanned by the $S$-local objects. This subcategory will be denoted by $\widehat{\Mod}^{I}_{A}$, or if $I$ is clear from context, simply by $\widehat{\Mod}_{A}$.
\end{definition}

\begin{observation} If $M$ is an $I$-adically complete $A$-module, then the $I$-adic filtration
	\[ ins^{0}(M) \otimes_{A} I^{\star}A \]
is in fact filtration complete. Indeed, the functor
	\[ I^{\star} \colon \Mod_{A} \to \fil \Mod_{A} \]
sends $A/I$-equivalences to $gr^{\star}$-equivalences, and the complete filtered modules are precisely the $gr^{\star}$-local filtered objects.

\end{observation}

\begin{warning} If $(A,I)$ is a discrete commutative ring equipped with an ideal, what we refer to as the $I$-adic filtration on a discrete $A$-module $M$ does not agree with the classical notion unless $I$ is locally generated by a regular sequence.
\end{warning}

\begin{definition} Let $A$ be a derived $\delta$-ring, and $I \to A$ a generalized Cartier divisor (a tensor invertible $A$-module equipped with a map of $A$-modules to $A$). The $\infty$-category of $(p,I)$-complete derived $\delta$-$A$-algebras is defined as the fiber product
	\[ \widehat{\dalg}(\Mod_{A}) := \dalg(\Mod_{A}) \times_{\Mod_{A}} \widehat{\Mod_{A}}. \]
\end{definition}

\begin{observation} The inclusions $\dalg(\Mod_{A}) \to \dalg(\Mod_{\Z_{p}})$ and $\widehat{\dalg}(\Mod_{A}) \to \dalg(\Mod_{A})$ preserve limits and sifted colimits. It follows formally that the forgetful functors $\dalg(\Mod_{A}) \to \Mod_{A}$ and $\widehat{\dalg}(\Mod_{A}) \to \widehat{\Mod}_{A}$ preserve limits and sifted colimits, and thus the resulting adjunctions are monadic. We will denote the left adjoints by $LSym^{\delta}_{A}$ and $\widehat{LSym}^{\delta}_{A}$ respectively.

\end{observation}


\begin{example} \label{cohomology is a delta ring} 

Recall from Proposition \ref{limits and colimits} the forgetful functor $\dalg(\Mod_{\Z_{p}}) \to CAlg(\Mod_{\Z_{p}})$ preserves limits. It follows that for any sheaf of (derived) $\delta$-rings $\mathcal{F}$ on a site $\mathcal{C}$, the derived global sections $R\Gamma(\mathcal{C}, \mathcal{F})$ can be canonically promoted to a derived $\delta$-ring.

In particular, given a prism $(A,I)$, and an $A/I$-algebra $R$, the derived prismatic cohomology $\Prism_{R/A}$ may be endowed with the structure of a $(p,I)$-complete derived $\delta$-$A$-algebra.

\end{example}

\begin{example} (perfect $\delta$-rings) \label{perfect delta rings} A derived $\delta$-ring $A$ is said to be \emph{perfect} if the Frobenius endomorphism 
	\[ \varphi_{A}: A \to A \]
of Proposition \ref{frobenius lifts on delta rings} is an isomorphism.

Perfect $p$-complete animated $\delta$-rings are always isomorphic to the Witt vectors of a perfect $\mathbb{F}_{p}$-algebra, and in particular are discrete since perfect $\mathbb{F}_{p}$-algebras are always discrete (Proposition 11.6 of \cite{Bhatt-Scholze2}). In the derived setting, there exist non-discrete perfect $\mathbb{F}_{p}$-algebras (see Remark 11.7 of \cite{Bhatt-Scholze2}).
Despite the increase in generality in the derived setting, the analogous characterization of perfect $\delta$-rings still holds, as stated below.
\end{example}

\begin{proposition} Let $R$ be a $p$-complete perfect derived $\delta$-ring. Then the canonical map $R \to W(R \otimes_{\mathbb{Z}_{p}} \mathbb{F}_{p})$ is an equivalence. 
\end{proposition}

\begin{proof}

The very argument which establishes this fact in the animated setting carries over to our more general setting. We will content ourselves with a brief sketch. Recall, the derived Witt vectors are defined as the right adjoint to the forgetful functor
	\[ \dalg(\Mod_{\Z_{p}}) \to \calg(\Mod_{\Z_{p}}). \]
It follows formally that upon restricting to $\mathbb{F}_{p}$-algebras, the derived Witt vector functor takes values in $p$-complete $\delta$-rings. For any perfect derived $\mathbb{F}_{p}$-algebra $S$, we declare
	\[ W_{n}(S) := W(S) \otimes_{\Z_{p}} \Z/p^{n}\Z, \]
which may be identified with the derived truncated Witt vectors of Example \ref{derived truncated Witt vectors} in this case. The classification of perfect $p$-complete derived $\delta$-rings then follows from the series of observations:

\begin{enumerate}

\item Denote by $\overline{R}$ the reduction mod $p$ of $R$: $\overline{R} := R \otimes_{\Z_{p}} \mathbb{F}_{p}$. Each projection $W_{n}(\overline{R}) \to W_{n-1}(\overline{R})$ is a square zero extension by $I_{n} := (p^{n-1})/(p^{n}) \otimes_{\Z_{p}} W_{n}(\overline{R})$.

\item The existence and uniqueness of deformations across these thickenings are controlled by $Ext^{i}(L_{\overline{R}/\mathbb{F}_{p}}, I_{n})$ for $i = 1, 2$ respectively. The cotangent complex for a map of derived commutative rings is studied in Section 4.4 of \cite{Raksit}.

\item Since $\overline{R}$ is perfect,
	\[ L_{\overline{R}/\mathbb{F}_{p}} \simeq 0 \]
for all $n$.

\item Combining the previous two facts, we see that
	\[ R \otimes_{\Z_{p}} \Z/p^{n}\Z \simeq W_{n}(\overline{R}) \]
from which $p$-completeness of both sides allows us to deduce that the canonical map
	\[ R \to W(\overline{R}) \]
is an isomorphism.

\end{enumerate}

\end{proof}

\begin{example} The following example was suggested to the author by Benjamin Antieau. Let $E$ be an ordinary elliptic curve over $\mathbb{F}_{p}$. Then $R\Gamma(E, \mathcal{O}_{E})$ is a perfect $\mathbb{F}_{p}$-algebra. By applying Witt vectors pointwise to the sheaf $\mathcal{O}_{E}$, one obtains a sheaf of $p$-complete $\delta$-rings $W\mathcal{O}_{E}$ on the elliptic curve $E$. The resulting $\delta$-structure on $R\Gamma(E, W\mathcal{O}_{E})$ is easily seen to be perfect, and since the derived global sections are automatically $p$-complete, Example \ref{perfect delta rings} implies that we have a canonical equivalence of derived $\delta$-rings:
	\[ R\Gamma(E, W\mathcal{O}_{E}) \simeq W(R\Gamma(E, W\mathcal{O}_{E}) \otimes_{\Z_{p}} \mathbb{F}_{p}) \simeq W(R\Gamma(E, W\mathcal{O}_{E} \otimes \mathbb{F}_{p})) \simeq W(R\Gamma(E, \mathcal{O}_{E})). \]

\end{example}


\section{Prismatic Cohomology}

Fix a $\delta$-ring $A$. The results of the previous section allow us to contemplate the free (filtered) $\delta$-$A$-algebra on an arbitrary (filtered) derived $A$-algebra. In particular, for any $A$-algebra $R$, we can apply this procedure to the Hodge-filtered derived infinitesimal cohomology $F^{\star}_{H}\mathbbl{\Pi}_{R/A}$. In the case where $(A,I)$ is a prism and $R$ is an $\overline{A} := A/I$-algebra, the Hodge filtered derived infinitesimal cohomology $F^{\star}_{H}\mathbbl{\Pi}_{R/A}$ is naturally an algebra over the $I$-adic filtration $I^{\star}A$ on $A$, and thus so is the free $\delta$-algebra thereon.

We may view
	\[\left(Free^{\delta}_{A}(\mathbbl{\Pi}_{R/A}), F^{1}_{H} Free^{\delta}_{A}(\mathbbl{\Pi}_{R/A})\right)\]
as a '$\delta$-pair' over $(A,I)$. Our task then becomes to formulate an appropriate notion of prism object in derived commutative rings, and to take the prismatic envelope of this $\delta$-pair. Following the \v{C}ech--Alexander approach to prismatic cohomology, we would then expect this prismatic envelope to compute prismatic cohomology, which we will see to be true.

The condition on prisms that the Cartier divisor be locally generated by a distinguished element is difficult to make sense of in the nonconnective setting. To work around this difficulty, we will instead appeal to the notion of rigidity of maps between prisms to define an appropriate analogue of prisms in the derived category. Recall, any map of prisms $(A,I) \to (B,J)$ satisfies the property that
	\[ I \otimes_{A} B \simeq J \]
(see Lemma 3.5 in \cite{Bhatt-Scholze}). Furthermore, this property almost characterizes prisms among those pairs $(B,J)$ where $B$ is a $(p,I)$-complete $\delta$-$A$-algebra, $J \subset B$ is an ideal, and the map $I \to B$ factors through $J$. In fact, the only further condition to check in order to guarantee $(B,J)$ is a prism is that $B$ is $I$-torsion free, a constraint which we will ignore in the setting of derived $\delta$-$A$-algebras. In particular, as soon as we specify that our prisms receive a map from $(A,I)$, the datum of the Cartier divisor is completely determined by $I$ itself, and we will take the $\infty$-category of $(p,I)$-complete $\delta$-$A$-algebras, $\compdalg(\Mod_{A})$, as our analogue of prisms. With this definition in play, the analogue of prismatic envelopes is manifest: it is left adjoint to the $I$-adic filtration functor.


\subsection{Recollections on Infinitesimal Cohomology}

In this subsection, we review some basic structural properties of derived infinitesimal cohomology following the perspective of \cite{Antieau}. Beyond the definition, the most important feature for us will be the notion of the vertical filtration (Theorem \ref{vertical filtration}) on the associated graded of the Hodge filtration. It will play an important technical role in Section 3.3 (see Construction \ref{conjugate filtration for theta} and Theorem \ref{descent for theta}).

\begin{definition} Recall that the functor
	\[ \calg(\fil \Mod_{A}) \xrightarrow{gr^{0}} \calg(\Mod_{A}) \]
admits a left adjoint (see \cite{Antieau}, as well as \cite{Raksit} for the analogous description of de Rham cohomology). Let $F^{\star}\mathbbl{\Pi}_{-/A}$ denote the left adjoint. We refer to this functor as \emph{Hodge-filtered derived infinitesimal cohomology}.
\end{definition}

It is shown in \cite{Antieau} that the associated graded of the Hodge filtration may be identified with
	\[ gr^{\star}_{H} \mathbbl{\Pi}_{R/A} \simeq LSym^{\star}_{R}(L_{R/A}[-1]). \]

\begin{example} \label{infinitesimal of regular sequence} Suppose $J \subset A$ is an ideal generated by a regular sequence. Then the Hodge filtered infinitesimal cohomology of $A/J$ with respect to $A$ may be identified
	\[ F^{\star}_{H}\mathbbl{\Pi}_{(A/J)/A} \simeq LSym^{\star}_{A}(J) \]
with the derived $J$-adic filtration on $A$. Note that this generally doesn't agree with the classical $J$-adic filtration in weights higher than 1 unless $J$ is flat.

\end{example}

Recall, for any triple of rings $A \to B \to C$, the Hodge-filtered de Rham cohomology $F^{\star}_{H} dR_{C/A}$ can be reconstructed from $F^{\star}_{H} dR_{C/B}$ and $F^{\star}_{H} dR_{B/A}$ by means of the Gauss-Manin connection. One incarnation of this structure is a refinement of the Hodge-filtration to a double filtration, whose graded pieces are expressible in terms of the cotangent complexes of $C/B$ and $B/A$. We now record the analogous construction for infinitesimal cohomology. 

Our first observation is that in the above context, $F^{\star}_{H} \mathbbl{\Pi}_{C/A}$ admits a more precise universal property. The map of $A$-algebras $B \to C$ yields a canonical map $F^{\star}_{H} \mathbbl{\Pi}_{B/A} \to F^{\star}_{H} \mathbbl{\Pi}_{C/A}$.

\begin{lemma} \label{infinitesimal cohomology of a triple} The functor
	\[ F^{\star}_{H}\mathbbl{\Pi}_{-/A} \colon \calg(\Mod_{B}) \to \calg(\fil \Mod_{F^{\star}_{H} \mathbbl{\Pi}_{B/A}}) \]
is left adjoint to $gr^{0}$.
\end{lemma}

\begin{proof} We can witness both the target and the source as slice categories:
	\[ \calg(\Mod_{B}) \simeq (\calg(\Mod_{A}))_{B \downarrow} \]
and
	\[ \calg(\fil \Mod_{F^{\star}_{H} \mathbbl{\Pi}_{B/A}}) \simeq (\calg(\fil \Mod_{A}))_{F^{\star}_{H} \mathbbl{\Pi}_{B/A} \downarrow}. \]
The claim thus follows the definition of $F^{\star}_{H} \mathbbl{\Pi}_{-/A}$ as a left adjoint along with Lemma \ref{adjunctions on slice cats} below.
\end{proof}

\begin{lemma} \label{adjunctions on slice cats} Let $\mathcal{C}$ and $\mathcal{D}$ be presentable $\infty$-categories. Let $F\colon \mathcal{C} \to \mathcal{D}$ be a left adjoint to a functor $G\colon \mathcal{D} \to \mathcal{C}$. Fix an object $x \in \mathcal{C}$, and consider the induced functor
	\[ F_{x}\colon \mathcal{C}_{x \downarrow} \to \mathcal{D}_{F(x) \downarrow}. \]
Fix a unit map $\eta\colon Id_{\mathcal{C}} \to G \circ F$. Then there exists a functor
	\[ G_{x}\colon \mathcal{D}_{F(x) \downarrow} \to \mathcal{C}_{x \downarrow} \]
satisfying
	\[ G_{x}( F(x) \xrightarrow{\alpha} y) = (x \xrightarrow{\eta_{x}} G(F(x)) \xrightarrow{G(\alpha)} G(y)) \]
and $G_{x}$ is right adjoint to $F_{x}$.
\end{lemma}

\begin{proof} This is (the categorical dual of) Proposition 5.2.5.1 in \cite{Lurie}.
\end{proof}

\begin{lemma} \label{infinitesimal base change} Let $A \to B \to C$ be a triple of derived commutative rings. Then there is a canonical equivalence of filtered derived $B$-algebras
	\[ F^{\star}_{H} \mathbbl{\Pi}_{C/A} \otimes_{F^{\star}_{H} \mathbbl{\Pi}_{B/A}} ins^{0}(B) \simeq F^{\star}_{H}\mathbbl{\Pi}_{C/B}. \]
\end{lemma}

\begin{proof} We identify the universal property of the left hand side with that of the right hand side. Fix a filtered $B$-algebra $F^{\star}R$. Then

\begin{align*}
Hom_{\calg(\fil \Mod_{B})}&(F^{\star}_{H} \mathbbl{\Pi}_{C/A}\otimes_{F^{\star}_{H} \mathbbl{\Pi}_{B/A}} ins^{0}(B), F^{\star}R)\\
	& \simeq Hom_{\calg(\fil \Mod_{F^{\star}_{H} \mathbbl{\Pi}_{B/A}})}(F^{\star}_{H}\mathbbl{\Pi}_{C/A}, F^{\star}R) \\
	& \simeq Hom_{\calg_{B}}(C, gr^{0}R) \\
	& \simeq Hom_{\calg(\fil \Mod_{B})}( F^{\star}_{H} \mathbbl{\Pi}_{C/B}, F^{\star}R)
\end{align*}
as desired.
\end{proof}

\begin{construction} Fix $F^{\star} R \in \calg(F^{\geq 0} \Mod_{k})$. For each $i \in \Z_{\geq 0}$ we associate a filtered $F^{\star} R$ module $F^{\geq i} R$ by the formula
	\[ F^{\geq i}R := ins^{i}(F^{i}R) \otimes_{ins^{0}(F^{0}R)} F^{\star}R. \]
Observe that this is concentrated in weights $\geq i$ and in weight $j \geq i$ agrees with $F^{j}R$.
\end{construction}

\begin{construction} Let $A \to B \to C$ be a triple of rings. We define a filtration on $F^{\star}_{H}\mathbbl{\Pi}_{C/A}$ by declaring
	\[ F_{KO}^{i}(F^{\star}_{H} \mathbbl{\Pi}_{C/A}) := F^{\star}_{H} \mathbbl{\Pi}_{C/A} \otimes_{F^{\star}_{H} \mathbbl{\Pi}_{B/A}} F^{\geq i}_{H} \mathbbl{\Pi}_{B/A}. \]
Here we are taking the tensor product in $\fil \Mod_{A}$. We refer to this as \emph{the Katz-Oda filtration} (following the terminology of \cite{Guo-Li}).
\end{construction}

\begin{notation}
The Katz-Oda filtration is evidently functorial in $C$, and yields a functor
	\[ F^{\star, \star}\mathbbl{\Pi}_{-/A} \colon \calg(\Mod_{B}) \to F^{\geq (0,0)} \calg(\Mod_{A}) \]
to bi-filtered derived commutative $A$-algebras, where the first index is recording the Katz-Oda filtration and the second the internal filtration. It is helpful to unwind the notation explicitly:
	\[ F^{i, \star} \mathbbl{\Pi}_{C/A} := F^{i}_{KO}F^{\star}_{H} \mathbbl{\Pi}_{C/A} \]
	\[ F^{i, j} \mathbbl{\Pi}_{C/A} := F^{j}( F^{\star}_{H}\mathbbl{\Pi}_{C/A} \otimes_{F^{\star}_{H} \mathbbl{\Pi}_{B/A}} F^{\geq i}_{H} \mathbbl{\Pi}_{B/A}) \]
\end{notation}

\begin{lemma}[compare with Lemma 3.13 in \cite{Guo-Li}] \label{double filtration} Let $A \to B \to C$ be a triple of rings. Then $F^{\star, \star} \mathbbl{\Pi}_{C/A}$ enjoys the following properties:

\begin{enumerate}

\item For any $j$, we have $F^{0,j} \mathbbl{\Pi}_{C/A} \simeq F^{j}_{H}\mathbbl{\Pi}_{C/A}$.
\item For each $0 \leq j \leq i$, we have an identification
	\[ F^{i,j}\mathbbl{\Pi}_{C/A} \simeq F^{0}_{H}\mathbbl{\Pi}_{C/A} \otimes_{F^{0}_{H}\mathbbl{\Pi}_{B/A}} F^{i} \mathbbl{\Pi}_{B/A} \simeq F^{i,0} \mathbbl{\Pi}_{C/A}. \]
\item For each $0 \leq i \leq j$, we have a natural identification
	\[ Cofib(F^{i+1, j} \mathbbl{\Pi}_{C/A} \to F^{i, j} \mathbbl{\Pi}_{C/A}) \simeq F^{j-i} \mathbbl{\Pi}_{C/B} \otimes_{B} LSym^{i}(L_{B/A}[-1]). \]
\end{enumerate}
\end{lemma}

\begin{proof} For (1): This follows immediately from our definition.

For (2): We can compute the tensor product defining $F^{i,j}\mathbbl{\Pi}_{C/A}$ via the Bar resolution:
\begin{center}
\begin{align*}
	 F^{i,j} \mathbbl{\Pi}_{C/A} & := F^{j}(F^{\star}_{H} \mathbbl{\Pi}_{C/A} \otimes_{F^{\star}_{H} \mathbbl{\Pi}_{B/A}} F^{\geq i}_{H}\mathbbl{\Pi}_{B/A}) \\
	  			& \simeq colim_{\Delta^{op}}\left( ... \rightthreearrow F^{j}(F^{\star}_{H}\mathbbl{\Pi}_{C/A} \otimes_{A} F^{\star}_{H}\mathbbl{\Pi}_{B/A} \otimes_{A} F^{\geq i}_{H}\mathbbl{\Pi}_{B/A}) \rightrightarrows F^{j}(F^{\star}_{H} \mathbbl{\Pi}_{C/A} \otimes_{A} F^{\geq i}_{H} \mathbbl{\Pi}_{B/A})\right)
\end{align*}
\end{center}
where the tensor products adorned by $A$ are given by Day convolution. Unpacking the Day convolution expresses each term in the simplicial object as a colimit:
\begin{align*}
	 F^{j}(F^{\star}_{H} \mathbbl{\Pi}_{C/A} \otimes_{A} (F^{\star}_{H} & \mathbbl{\Pi}_{B/A})^{\otimes_{A} k}  \otimes_{A} F^{\geq i}_{H}\mathbbl{\Pi}_{B/A}) \\
	& \simeq colim_{n+m \geq j} F^{n}_{H} \mathbbl{\Pi}_{C/A} \otimes_{A} F^{m}\left( (F^{\star}_{H} \mathbbl{\Pi}_{B/A})^{\otimes_{A} k} \otimes_{A} F^{\geq i}_{H}\mathbbl{\Pi}_{B/A}\right) \\
	& \simeq F^{0}_{H}\mathbbl{\Pi}_{C/A} \otimes_{A} (F^{0}_{H} \mathbbl{\Pi}_{B/A})^{\otimes_{A} k} \otimes_{A} F^{i}_{H}\mathbbl{\Pi}_{B/A}
\end{align*}
where the colimit collapses thanks to the hypothesis that $j \leq i$ (and $F^{\geq i}_{H} \mathbbl{\Pi}_{B/A}$ is concentrated in weights $\geq i$). Inputting this back in to the Bar resolution yields the claim.

For (3): By definition we have
	\[ gr^{i}_{KO}(F^{\star}_{H}\mathbbl{\Pi}_{C/A}) \simeq F^{\star}_{H}\mathbbl{\Pi}_{C/A} \otimes_{F^{\star}_{H} \mathbbl{\Pi}_{B/A}} LSym^{i}(L_{B/A}[-1]). \]	
We can now appeal to Lemma \ref{infinitesimal base change} to rewrite the right hand side as
\begin{align*}
 F^{\star}_{H}\mathbbl{\Pi}_{C/A} \otimes_{F^{\star}_{H} \mathbbl{\Pi}_{B/A}} LSym_{B}^{i}(L_{B/A}[-1]) & \simeq (F^{\star}_{H} \mathbbl{\Pi}_{C/A} \otimes_{F^{\star}_{H} \mathbbl{\Pi}_{B/A}} B) \otimes_{B} LSym_{B}^{i}(L_{B/A}[-1]) \\
 	& \simeq F^{\star}_{H} \mathbbl{\Pi}_{C/B} \otimes_{B} LSym^{i}(L_{B/A}[-1]).
\end{align*}
where $LSym^{i}(L_{B/A}[-1])$ has weight $i$. The claim now follows by weight considerations.
\end{proof}

\begin{construction} Keeping with the notation of the preceding lemma, let us denote by 
	\[ X^{i,j} := Cone(F^{i, j+1} \mathbbl{\Pi}_{C/A} \to F^{i,j} \mathbbl{\Pi}_{C/A}). \]
Observe that we have natural maps
	\[ X^{i,j} \xrightarrow{\alpha_{i,j}} X^{i-1, j} \]
arising from the structure maps of the filtration.
In addition, the maps $F^{i,j}\mathbbl{\Pi}_{C/A} \to F^{0,j}\mathbbl{\Pi}_{C/A} \simeq F^{j}_{H}\mathbbl{\Pi}_{C/A}$ induce canonical maps
	\[ X^{i,j} \to gr^{j}_{H} \mathbbl{\Pi}_{C/A}. \]
We define the \emph{vertical filtration} on $gr^{\star}_{H} \mathbbl{\Pi}_{C/A}$ by
	\[ F_{n}^{vert}(gr^{\star}_{H} \mathbbl{\Pi}_{C/A}) := \bigoplus_{k} X^{k, k+n}. \]
The structure maps are given by the direct sum of the $\alpha_{k, k+n}$, which yields an increasing filtration.

\end{construction}

\begin{theorem}[compare with Lemma 4.7 in \cite{Li-Liu}] \label{vertical filtration} The vertical filtration endows $gr^{\star}_{H} \mathbbl{\Pi}_{C/A}$ with the structure of a derived filtered algebra object in graded modules over $gr^{\star}_{H}(\mathbbl{\Pi}_{B/A})$ - i.e. as an object in the category $ \calg(F_{\geq 0} (\gr \Mod_{gr^{\star}_{H} \mathbbl{\Pi}_{B/A}}))$. Furthermore, the filtration is exhaustive and the associated graded is given by
	\[ gr_{i}^{vert}(gr^{\star}_{H} \mathbbl{\Pi}_{C/A}) \simeq LSym^{i}_{C}(L_{C/B}[-1]) \otimes_{B} LSym_{B}^{\star}(L_{B/A}[-1]). \]

\end{theorem}

\begin{proof} The filtration is exhaustive by inspection and the identification of the associated graded follows immediately by appealing to Lemma \ref{double filtration} (3), to identify the cone of the $\alpha_{i,j}$:
	\[ Cone(\alpha_{i,j}) \simeq LSym_{C}^{j-i +1}(L_{C/B}[-1]) \otimes_{B} LSym^{i-1}_{B}(L_{B/A}[-1]). \]	
It remains to be seen that the filtration may be endowed with a derived algebraic structure. We will denote by
	\[ gr^{(-,\star)} \colon F^{\geq (0,0)}\Mod_{A} \to \fil \gr \Mod_{A} \]
the functor which takes a doubly filtered object and passes to the associated graded in the second weight. This is obviously a morphism of derived algebraic contexts, and the vertical filtration is obtained by a simple reindexing of this construction. Namely, consider the symmetric monoidal functor
	\[ p\colon \Z_{\geq 0}^{op} \times \Z_{\geq 0}^{disc} \xrightarrow{ (i,k) \to (k, k+i)} \Z_{\geq 0} \times \Z_{\geq 0}^{disc}. \]
Pullback along $p$ yields a morphism of derived algebraic contexts $F^{\geq 0} \gr \Mod_{A}\to  F_{\geq 0} \gr \Mod_{A}$, and
	\[ F^{v}_{\star}(gr^{\star} \mathbbl{\Pi}_{R/A}) \simeq p^{\star} gr^{(-,\star)}(F^{\star}_{KO} F^{\star}_{H} \mathbbl{\Pi}_{R/A}) \]
from which we conclude.
\end{proof}


\subsection{$I$-adic Envelopes}

Fix a derived algebraic context $\mathcal{C}$ and an algebra $F^{\star} A \in  \calg(\pair \mathcal{C})$.

\begin{notation} We view $F^{\star} A$ as a ring equipped with a generalized ideal, and adopt the following conventions which mimic standard notation in prismatic cohomology.

\begin{itemize} 

\item We will denote by $I := ev^{1}(F^{\star}A)$, $A := ev^{0}(F^{\star}A)$, and often write $(A,I)$ in place of $F^{\star}A$.

\item We will denote by $I^{\star}A \in \calg(F^{\geq 0}\mathcal{C})$ the $I$-adic filtration associated to $(A,I)$ as defined in Notation \ref{I-adic filtration}.

\item We will denote by $\overline{A} := gr^{0}(F^{\star}A) \simeq cofib(I \to A)$.

\item Given an $\overline{A}$-module $M$, we will denote by $M\{n\}$ the tensor product
	\[ M \otimes_{\overline{A}} gr^{n}(I^{\star}A). \]

\end{itemize}
\end{notation}

\begin{definition} \label{I-adic} 
We define the \emph{$I$-adic filtration functor}
	\[ I^{\star} \colon \calg(\Mod_{A}) \to \calg(\fil \Mod_{I^{\star}A}) \]
as the composite
	\[ \calg(\Mod_{A}) \xrightarrow{ins^{0}} \calg(\fil \Mod_{A}) \xrightarrow{ - \otimes_{A} I^{\star}A} \calg(\fil \Mod_{I^{\star}A}). \]

\end{definition}

\begin{definition} An algebra $(A,I) \in \calg(\pair\mathcal{C})$ is said to be a \emph{Cartier pair} if $I$ is a tensor invertible $A$-module.

\end{definition}

\begin{lemma} Let $(A,I)$ be a Cartier pair. Then the functor $I^{\star}$ of \ref{I-adic} preserves small limits and colimits. In particular, it admits a left adjoint.

\end{lemma}

\begin{proof} Since the evaluation functors
	\[ ev_{i}\colon  \calg(\fil \Mod_{A}) \to \Mod_{A} \]
are jointly conservative and preserve small limits and colimits, it suffices to check that $ev_{i} \circ I^{\star}$ preservers small limits and colimits. But 
\[ev_{i} \circ I^{\star} \simeq - \otimes_{A} LSym_{A}^{i}(I) \simeq - \otimes_{A} I^{\otimes_{A} i}\]
is given by tensoring with a tensor invertible $A$-module, from which the claim follows.
\end{proof}

\begin{definition} The left adjoint to $I^{\star}$ is denoted by $\calgenv_{I}$. We will refer to this adjoint as the (uncompleted) \emph{derived $I$-adic envelope}. We may also restrict attention to the $\infty$-category of $(p,I)$-complete $A$-algebras, $\widehat{\calg}(\Mod_{A})$, in which case the $I$-adic filtration functor factors (by definition) through the $\infty$-category of complete filtered algebras, $ \calg(\widehat{\fil} \Mod_{A})$. In this case, we denote the left adjoint by $\compcalgenv_{I}$, and still refer to it as the derived $I$-adic envelope.

\end{definition}

\begin{variant} \label{prismatic envelope} Suppose $(A,I)$ is a Cartier pair such that the underlying algebra $A$ is equipped with a $\delta$-structure. The same construction carries through for derived $\delta$-$A$-algebras, yielding a left adjoint
	\[ \env_{I}\colon \dalg(\fil \Mod_{A}) \to \dalg(\Mod_{A}) \]
to the $I$-adic filtration functor.	We will still refer to this as the derived $I$-adic envelope, but distinguish it notationally as depicted. 
\end{variant}

It will also be useful to study the analogue of envelopes on graded algebras.

\begin{variant} \label{graded envelope} The functor
	\[  \calg(\Mod_{\overline{A}}) \xrightarrow{ins^{0}(-) \otimes_{\overline{A}} gr^{\star}(I^{\star}A)} \calg(\gr \Mod_{gr^{\star}(I^{\star}A)}) \]
also admits a left adjoint.

Indeed, the base change functor on graded algebras preserves all limits. One can check such a claim pointwise on the level of underlying modules, where we see that
	\[ (M^{\star} \otimes_{A/I} gr^{\star}(I^{\star}A))^{n} = \bigoplus_{i+j = n} M^{i} \otimes_{A/I} I^{j}/I^{j+1}. \]
Since the direct sum is finite and $I^{i}/I^{i+1}$ is an invertible $A/I$-module, the claim follows.

 We will denote the left adjoint by $\grenv_{I}$, and refer to this as the \emph{graded $I$-adic envelope}.

\end{variant}

The relationship between $\calgenv_{I}$ and $\grenv_{I}$ may be expressed as follows.

\begin{lemma} \label{graded vs filtered envelopes} The composite 
	\[ \calg(\fil \Mod_{I^{\star}A}) \xrightarrow{\calgenv_{I}} \calg(\Mod_{A}) \xrightarrow{- \otimes_{A} \overline{A}} \calg(\Mod_{\overline{A}}) \]
admits a canonical factorization:

\begin{center}
\begin{tikzcd}

 \calg(\fil \Mod_{I^{\star}A}) \arrow[r, "\calgenv_{I}"] \arrow[d, "gr^{\star}"] 
		& \calg(\Mod_{A}) \arrow[d, "- \otimes_{A} \overline{A}"] \\
 \calg(\gr \Mod_{gr^{\star}(I^{\star}A)}) \arrow[r, "\grenv_{I}"]
		& \calg(\Mod_{\overline{A}})

\end{tikzcd}
\end{center}
\end{lemma}

\begin{proof}Recall that the associated graded functor
	\[ gr^{\star}\colon \calg(\fil \Mod_{A}) \to \calg(\gr \Mod_{A}) \]
admits a right adjoint which sends a graded algebra $R^{\star}$ to the filtered object 
	\[ ... \xrightarrow{0} R^{n} \xrightarrow{0} R^{n-1} \xrightarrow{0} ... \]
We will denote this right adjoint by $\beta$ for the remainder of this proof.

Unwinding definitions, one verifies that the following diagram of functors commutes:

\begin{center}
\begin{tikzcd}
\calg(\Mod_{A}) \arrow[r, "ins^{0}"]
	&\calg(\fil \Mod_{\overline{A}}) \arrow[r]
		& \calg(\fil \Mod_{A}) \arrow[r, " - \otimes_{A} I^{\star}A"] 
	 		& \calg(\fil \Mod_{I^{\star}A}) \\
\calg(\Mod_{A}) \arrow[r, "ins^{0}"] \arrow[u, "Id"]
	& \calg(\gr \Mod_{\overline{A}}) \arrow[u, "\beta"] \arrow[rr, "- \otimes_{\overline{A}} gr^{\star}(\overline{A})"]
			&	& \calg(\gr \Mod_{gr^{\star}(I^{\star}A)}) \arrow[u, "\beta"]

\end{tikzcd}
\end{center}
from which the claim follows by passing to left adjoints.
\end{proof}

We now establish a basic naturality result for envelopes which will be used in the next section.

\begin{observation} Given a map $f\colon \mathcal{C} \to \mathcal{D}$ of derived algebraic contexts, we can promote $f$ to a colimit-preserving map
	\[ \hat{f} \colon \calg(\pair \mathcal{C}) \to\calg(\pair \mathcal{D}) \]
in the obvious way. Given a Cartier pair $(A,I)$ over $\mathcal{C}$, $\hat{f}(A,I)$ is also a Cartier pair.

\end{observation}

It is unclear whether or not envelopes commute with the functor $f$ in the above generality, but we can say something about the case where $f$ is the functor of taking a filtered object to its associated graded.

\begin{proposition} \label{naturality for envelopes} Let $(A,I)$ be a Cartier pair in $\mathcal{C}' := \fil \mathcal{C}$, and let $(B,J) := (gr^{\star}(A), gr^{\star}(I))$ be the Cartier pair in $\mathcal{D} := \gr \mathcal{C}$ associated to the map of derived algebraic contexts $gr^{\star}\colon \mathcal{C}' \to \mathcal{D}$. Then the following diagram commutes

\begin{center}
\begin{tikzcd}

\calg(\fil \Mod_{I^{\star}A}) \arrow[r, "\calgenv_{I}"] \arrow[d, "\fil gr^{\star}"] 
	& \calg(\Mod_{A})  \arrow[d, "gr^{\star}"] \\
\calg(\fil \Mod_{J^{\star}B}) \arrow[r, "\calgenv_{J}"]
	& \calg(\Mod_{B})

\end{tikzcd}
\end{center}
The analogous diagram for graded envelopes also commutes.

\end{proposition}

\begin{proof} Since we understand the right adjoint to $gr^{\star}$ (and thus to $\fil gr^{\star}$), we just check that the diagram of right adjoints commutes, which follows by exactly the same reasoning as employed in Lemma \ref{graded envelope}.
\end{proof}

\begin{observation} Given any derived $A$-algebra, $B$, the co-unit
	\[ \calgenv_{I}(I^{\star}B) \to B \]
is an equivalence, since $I^{\star}$ is fully faithful.
\end{observation}

\begin{observation} \label{envelopes of 01 filtered} We could make an analogous definition of envelopes for $\{0,1\}$-filtered algebras instead of general filtered algebras, but the distinction between these two constructions is irrelevant for algebras whose filtration is induced from weight 1. Indeed, there is a canonical commutative diagram

\begin{center}
\begin{tikzcd}

\calg(\fil \Mod_{I^{\star}A}) \arrow[r, "\calgenv_{I}"] 
	& \calg(\Mod_{A}) \\
\calg(\pair \Mod_{(A,I)}) \arrow[ur, "\calgenv^{\{0,1\}}_{I}"] \arrow[u, "p_{!}"]
\end{tikzcd}
\end{center}
through the category of $(A,I)$-algebras. To see this, observe that the associated diagram of right adjoints also factors through $\{0,1\}$-filtered rings since the $I$-adic filtration on $A$ is defined as a filtration induced from the generalized Cartier divisor $I$. Many of the examples of interest will have filtration freely induced by weight 1 (the main example being the Hodge filtration on infinitesimal cohomology), and in such cases we will not distinguish between these two notions of envelope.

\end{observation}

\begin{lemma} \label{graded envelope of free algebras} Fix a Cartier pair $(A,I)$ and let $R$ be an $\overline{A}$-algebra. For any $M \in \Mod_{R}$, there is a canonical equivalence
	\[ \grenv_{I}\left(LSym_{R}^{Gr^{\geq0}}(M(n)) \otimes_{\overline{A}} gr^{\star}(I^{\star}A)\right) \simeq LSym_{R}(M\{-n\}) \]
where $\grenv_{I}$ is the restriction of the graded envelope to $R$-algebras. That is, $\grenv_{I}$ is left adjoint to
	\[ \calg(\Mod_{R}) \xrightarrow{ ins^{0}(-) \otimes_{\overline{A}(0)} gr^{\star}(I^{\star}A)} \calg(\gr \Mod_{R(0) \otimes_{\overline{A}(0)} gr^{\star}(I^{\star}A)}). \]

\end{lemma}

\begin{proof} We will only establish the claim in the case that $R = \overline{A}$ and observe that the same proof carries through in general (with slightly more burdensome notation). We identify the universal properties of both sides. Fix an $\overline{A}$-algebra $S \in \calg_{\overline{A}}$. Recalling that $ins^{0}$ is right adjoint to $ev^{0}$ on graded objects, we obtain
\begin{align*}
& Hom_{\calg(\Mod_{\overline{A}})}(\grenv_{I}(LSym_{\overline{A}}^{Gr^{\geq0}}(M(n)) \otimes_{\overline{A}} gr^{\star}(I^{\star}A)), S)  \\
 & \simeq Hom_{Gr^{\geq0}\calg(\Mod_{gr^{\star}(I^{\star}A)})}( LSym_{\overline{A}}^{Gr^{\geq0}}(M(n)) \otimes_{\overline{A}} gr^{\star}(I^{\star}A), S(0) \otimes_{\overline{A}} gr^{\star}(I^{\star}A)) \\
 & \simeq Hom_{Gr^{\geq0}\Mod_{\overline{A}}}( M(n), S(0) \otimes_{\overline{A}} gr^{\star}(I^{\star}A))
\end{align*}
where we have used the definition of $\grenv_{I}$ and $LSym_{\overline{A}}^{Gr^{\geq0}}$ as left adjoints to rewrite the mapping spaces in the first and second equivalences respectively. We now appeal to the fact that $ev^{n}$ is also right adjoint to $ins^{n}$ to identify the preceding mapping space with $Hom_{\Mod_{\overline{A}}}(M, S \otimes_{\overline{A}} I^{n}/I^{n+1})$. Proceeding, we obtain
\begin{align*}
 Hom_{\Mod_{\overline{A}}}(M, S \otimes_{\overline{A}} I^{n}/I^{n+1}) 
 & \simeq Hom_{\Mod_{\overline{A}}}(M\{-n\}, S) \\
 & \simeq Hom_{\calg(\Mod_{\overline{A}})}(LSym_{\overline{A}}(M\{-n\}), S)
\end{align*}
as desired.
\end{proof}

We now turn our attention towards Variant \ref{prismatic envelope}, and relate the theory of envelopes in this case to the prismatic envelopes of \cite{Bhatt-Scholze}.

\begin{observation} \label{base change for envelopes} Given a map of prisms $(A,I) \to (B,J)$, we have a canonical equivalence of functors
	\[ \env_{I}(-) \otimes_{A}B \simeq \env_{J}( - \otimes_{I^{\star} A} J^{\star} B). \]
To see this, it suffices to show that the associated diagram of \emph{right} adjoints commutes:

\begin{center}
\begin{tikzcd}

\dalg(\Mod_{B}) \arrow[r] \arrow[d, "J^{\star}"]
	& \dalg(\Mod_{A}) \arrow[d, "I^{\star}"] \\
\dalg(\fil \Mod_{J^{\star}B}) \arrow[r]
	& \dalg(\fil \Mod_{I^{\star}A})

\end{tikzcd}
\end{center}
which is true by rigidity of maps of prisms. Indeed, for any $\delta$-$B$-algebra $R$, we see that
	\[ J^{\star}R := ins^{0}(R) \otimes_{B} J^{\star}B \simeq ins^{0}(R) \otimes_{B} (ins^{0}(B) \otimes_{A} I^{\star}A) \simeq I^{\star}R. \]

\end{observation}

\begin{construction} Fix a $\delta$-Cartier pair $(A,I)$. Given a complex $N \in \Mod_{A}$, we define a map of $\delta$-$A$-algebras
	\[ \gamma_{N}\colon \dsym_{A}(N) \to \dsym_{A}(N \otimes_{A} I^{-1}) \]
as the map induced by the map of $A$-modules
	\[ N \simeq I \otimes_{A} N \otimes_{A} I^{-1} \xrightarrow{id \otimes i} I \otimes_{A} \dsym_{A}(N \otimes_{A} I^{-1}) \xrightarrow{\mu} \dsym_{A}(N \otimes_{A} I^{-1}). \]
Here, $\mu$ is the multiplication arising from the $A$-algebra structure on $\dsym_{A}(N)$.
\end{construction}

\begin{example} Suppose $(A,I)$ is a prism and $I = (d)$ is endowed with an orientation. The free $\delta$-$A$-algebra on a single generator is then given by $\dsym(A) \simeq A\{x\}$ and the map $\gamma_{A}$ of the above construction may be identified with
	\[ A\{x\} \xrightarrow{ x \to d\cdot y} A\{y\}. \]

\end{example}

\begin{lemma} \label{envelopes of free algebras} Let $(A,I)$ be a $\delta$-Cartier pair. Fix an object $(N \to M) \in F^{\{0,1\}}\Mod_{A}$, and denote by $\dsym_{A}(M)\{\frac{N}{I}\}$ the pushout in the category of $\delta$-$A$-algebras

\begin{center}
\begin{tikzcd}
\dsym_{A}(N) \arrow[r, "\gamma_{N}"] \arrow[d]
	& \dsym_{A}(N \otimes_{A} I^{-1}) \arrow[d] \\
\dsym_{A}(M) \arrow[r]
	& \dsym_{A}(M)\{\frac{N}{I}\}.
\end{tikzcd}
\end{center}
Then there is a canonical equivalence of $\delta$-$A$-algebras
	\[ \env_{I}(\fdsym_{A}(N \to M)\otimes_{A}I^{\star}A) \xrightarrow{\simeq} \dsym_{A}(M)\{\frac{N}{I}\}. \]

\end{lemma}

\begin{proof} Fix a $\delta$-$A$-algebra $R$, and denote by $(B,J)$ the filtered $\delta$-algebra 
\[ \fdsym_{A}(N \to M)\otimes_{A}I^{\star}A.\]
It suffices to construct a functorial identification of mapping spaces
	\[ Hom_{\dalg(\Mod_{A})}(\dsym_{A}(M)\{\frac{N}{I}\}, R) \xrightarrow{\simeq} Hom_{\dalg(\fil \Mod_{I^{\star}A})}((B,J), I^{\star}R). \]
The definition of $\dsym_{A}(M)\{\frac{N}{I}\}$ as a pushout identifies $Hom_{\dalg( \Mod_{A})}(\dsym_{A}(M)\{\frac{N}{I}\}, R)$ as
\begin{align*}
Hom_{\dalg(\Mod_{A})}(\dsym_{A}(M), R) & \times_{Hom_{\dalg(\Mod_{A})}(\dsym_{A}(N), R)} Hom_{\dalg(\Mod_{A})}(\dsym(N\otimes_{A} I^{-1}), R) \\
	& \simeq Hom_{\Mod_{A}}(M, R) \times_{Hom_{\Mod_{A}}(N, R)} Hom_{\Mod_{A}}(N\otimes I^{-1}, R) \\
	& \simeq Hom_{\Mod_{A}}(M,R) \times_{Hom_{\Mod_{A}}(N,R)} Hom_{\Mod_{A}}(N, I \otimes R) \\
	& \simeq Hom_{F^{\{0,1\}}\Mod_{A}}((N\to M), I^{\star}R) \\
	& \simeq Hom_{\dalg(\fil \Mod_{A})}(\fdsym_{A}(N \to M), I^{\star}R) \\
	& \simeq Hom_{\dalg(\fil \Mod_{I^{\star}A})}((B,J), I^{\star}R) 
\end{align*}
as desired. Note that the third equivalence follows from the definition of mapping spaces in the filtered derived category.
\end{proof}

\begin{corollary} \label{envelopes of regular sequences} Let $(A,I)$ be a prism. Fix an ideal $J = (I,x_{1},...,x_{n})$ where the $x_{i}$ form a $(p,I)$-completely regular sequence in $A$. Then the canonical map
	\[ \compenv_{I}(J \to A) \xrightarrow{ \simeq } A\{\frac{J}{I}\}^{\wedge}_{(p,I)} \]
is an equivalence, where $A\{\frac{J}{I}\}^{\wedge}_{(p,I)}$ is the prismatic envelope of \cite{Bhatt-Scholze}.

\end{corollary}

\begin{proof} By Lemma \ref{envelopes of free algebras} and Remark 2.7 in \cite{Li-Liu}, it suffices to exhibit the filtered $\delta$-ring $(J \to A)$ as the pushout in $(p,I)$-complete filtered $\delta$-$A$-algebras of the following diagram:
\begin{center}
\begin{tikzcd}

I^{\star}A\{\hat{x}_{1},..., \hat{x}_{n}\} \arrow[r] \arrow[d, "\hat{x}_{i} \to x_{i}"]
	& \left( (I, \hat{x}_{1}, ..., \hat{x}_{n}) \to A\{\hat{x}_{1},..., \hat{x}_{n}\} \right)\\
I^{\star}A 
\end{tikzcd}
\end{center}
Since each term in the pushout is free, it suffices to identify the pushout of the following diagram in filtered derived $A$-algebras:
\begin{center}
\begin{tikzcd}
I^{\star}A[\hat{x}_{1},..., \hat{x}_{n}] \arrow[r] \arrow[d, "\hat{x}_{i} \to x_{i}"]
	& \left( (I, \hat{x}_{1}, ..., \hat{x}_{n}) \to A[\hat{x}_{1},..., \hat{x}_{n}] \right) \\
I^{\star}A
\end{tikzcd}
\end{center}
Each vertex of this diagram can be computed as the pushout of the corresponding row in the following diagram: 
\begin{center}
\begin{tikzcd}
F^{\star}_{H}\mathbbl{\Pi}_{\overline{A}/A}
	& F^{\star}_{H}\mathbbl{\Pi}_{A[\hat{x}_{1},..., \hat{x}_{n}]/A} \arrow[r] \arrow[l]
		& A[\hat{x}_{1},..., \hat{x}_{n}] \\
F^{\star}_{H}\mathbbl{\Pi}_{\overline{A}[\hat{x}_{1},..., \hat{x}_{n}]/A} \arrow[u] \arrow[d]
	& F^{\star}_{H}\mathbbl{\Pi}_{A[\hat{x}_{1},..., \hat{x}_{n}]/A} \arrow[r] \arrow[l] \arrow[u] \arrow[d]
		& A[\hat{x}_{1},..., \hat{x}_{n}] \arrow[u] \arrow[d] \\
F^{\star}_{H}\mathbbl{\Pi}_{\overline{A}/A}
	& 0 \arrow[r] \arrow[l]
		& 0
\end{tikzcd}
\end{center}
Indeed, appealing to Lemma \ref{infinitesimal base change} and Example \ref{infinitesimal of regular sequence}, the pushout of the first row is $F^{\star}_{H} \mathbbl{\Pi}_{\overline{A}/A[\hat{x}_{1},..., \hat{x}_{n}]} \simeq ((I, \hat{x}_{1},..., \hat{x}_{n}) \to A[\hat{x}_{1},..., \hat{x}_{n}])$ whereas the pushout of the second is $F^{\star}_{H}\mathbbl{\Pi}_{\overline{A}[\hat{x}_{1},..., \hat{x}_{n}]/A[\hat{x}_{1},..., \hat{x}_{n}]} \simeq I^{\star}A[\overline{x}_{1},..., \overline{x}_{n}]$.

The colimit of this larger diagram can be computed by first pushing out along the rows and then taking the colimit of the resulting span, or first pushing out the columns and then taking the colimit of the resulting span (see Lemma 1.13 in \cite{Devalapurkar-Haine}). The first procedure thus yields the colimit of the diagram we are interested in analyzing

The pushout of the columns results in applying $F^{\star}_{H} \mathbbl{\Pi}_{-/A}$ to the pushout diagram
\begin{center}
\begin{tikzcd}
\overline{A}[\hat{x}_{1},..., \hat{x}_{n}] \arrow[r, "\hat{x}_{i} \to 0"] \arrow[d, "\hat{x}_{i} \to x_{i}"] 
	& \overline{A} \\
\overline{A}
\end{tikzcd}
\end{center}
and by regularity of the sequence $x_{i}$, the pushout of this diagram is $A/J$. In particular Example \ref{infinitesimal of regular sequence} implies that the colimit of the big diagram is precisely $(J \to A)$, as desired.
\end{proof}


\subsection{Derived Prismatic Cohomology}

We now formulate the universal property of derived prismatic cohomology. Throughout this section $(A,I)$ will always denote a prism unless explicitly indicated otherwise.

\begin{observation} \label{prismatic cohomology} Let $(A,I)$ be a $\delta$-Cartier pair. Since $I \to A$ is a Cartier divisor, the composite
	\[ \compdalg(\Mod_{A}) \xrightarrow{Forget} \widehat{\calg}(\Mod_{A}) \xrightarrow{ - \otimes_{A} \overline{A}} \widehat{\calg}(\Mod_{\overline{A}}) \]
admits a left adjoint.

\end{observation}

\begin{definition} Let $(A,I)$ be any $\delta$-Cartier pair. We denote by
	\[ L\Prism_{-/A}\colon \compcalg(\Mod_{\overline{A}}) \to \compdalg(\Mod_{A}) \]
the left adjoint of \ref{prismatic cohomology}. Given $R \in \compdalg(\Mod_{\overline{A}})$, we refer to the object $L\Prism_{R/A}$ as the \emph{derived prismatic cohomology} of $R$ relative to $A$. This terminology is justified by Theorem \ref{main comparison} below.

\end{definition}

\begin{variant} Similarly, the functor
	\[ \compcalg(\Mod_{A}) \xrightarrow{ - \otimes_{A} \overline{A}} \compcalg(\Mod_{\overline{A}}) \]
admits a left adjoint, which we will denote by $L\Theta_{-/A}$. Observe that for any $\overline{A}$-algebra $R$, there is a canonical equivalence
	\[ L\Prism_{R/A} \simeq Free^{\delta}_{A} \circ L \Theta_{R/A}. \]

\end{variant}

\begin{example} Both $L\Theta_{R/A}$ and $L \Prism_{R/A}$ can be made slightly more explicit in the case that $R = LSym_{\overline{A}}(\overline{A}) \simeq \overline{A}\left<x\right>$ is a $p$-complete polynomial algebra over $\overline{A}$. For simplicity, we will assume that $I=(d)$ is principal.

Indeed, in this case we have a commutative diagram of right adjoints
\begin{center}
\begin{tikzcd}
\compcalg(\Mod_{\overline{A}}) \arrow[d]
	& \compcalg(\Mod_{A}) \arrow[l, "- \otimes_{A} \overline{A}"] \arrow[d] \\
 \widehat{\Mod_{\overline{A}}} 
 	& \widehat{\Mod_{A}} \arrow[l, " - \otimes_{A} \overline{A}"]
\end{tikzcd}
\end{center}
which implies that, denoting by $L$ the left adjoint to the bottom horizontal arrow, we have
	\[ L\Theta_{R/A} \simeq LSym_{A}(L(\overline{A})). \]
It thus suffices to identify $L(\overline{A})$, which is obtained by appealing to universal properties:
\begin{align*}
Hom_{\widehat{\Mod_{A}}}(L(\overline{A}), M) & \simeq Hom_{\widehat{\Mod_{\overline{A}}}}(\overline{A}, M \otimes_{A} \overline{A}) \\
								& \simeq cofib(Hom_{\widehat{\Mod_{A}}}(A, M) \xrightarrow{d^{*}} Hom_{\widehat{\Mod_{A}}}(A, M)) \\
								& \simeq Hom_{\widehat{\Mod_{A}}}(fib(d), M)
\end{align*}
where $d$ indicates the multiplication map $d: A \to A$. In particular, $L(\overline{A}) \simeq \overline{A}[-1]$ and we see that
	\[ L\Theta_{\overline{A}\left<x\right>/A} \simeq LSym_{A}(\overline{A}[-1]) \]
and
	\[ L\Prism_{\overline{A}\left<x\right>/A} \simeq \dsym_{A}(\overline{A}[-1]). \]

\end{example}

The starting point for our investigation into derived prismatic cohomology is the following expression of $L\Theta_{-/A}$.
\begin{lemma} There is a natural equivalence of functors
	\[ L\Theta_{-/A} \simeq \compcalgenv_{I} \circ F^{\star}_{H} \widehat{\mathbbl{\Pi}}_{-/A} \]
where the symbol $\widehat{\mathbbl{\Pi}}$ refers to Hodge-complete infinitesimal cohomology.

\end{lemma}

\begin{proof} The functor $ - \otimes_{A} \overline{A}$ factors as
	\[ \compcalg(\Mod_{A}) \xrightarrow{I^{\star}} \calg(\widehat{\fil} \Mod_{I^{\star}A}) \xrightarrow{gr^{0}} \calg(\Mod_{\overline{A}}). \]
Recall from Lemma \ref{infinitesimal cohomology of a triple} that the left adjoint to $gr^{0}$ is precisely $F^{\star}_{H} \widehat{\mathbbl{\Pi}}_{-/A}$. The Lemma now follows from the definition of $\compcalgenv_{I}$ as the left adjoint to $I^{\star}$. 
\end{proof}

\begin{observation} Following Example \ref{cohomology is a delta ring}, we see that for any smooth $A/I$-algebra $R$, $\Prism_{R/A}$ may be viewed as an object of $\compdalg(\Mod_{A})$. Furthermore, the Hodge-Tate complex $\overline{\Prism}_{R/A} \simeq \Prism_{R/A} \otimes_{A} \overline{A}$ is naturally an $R$-algebra, by definition.
By adjunction, we thus obtain a canonical map of derived $\delta$-rings
	\[ comp_{R/A}\colon L\Prism_{R/A} \to \Prism_{R/A}. \]
By left Kan extension, we obtain a comparison map for any animated $\overline{A}$-algebra $R$ (where the right hand side is now understood to be the left Kan extension of the usual prismatic cohomology).	
\end{observation}

\begin{theorem} \label{main comparison} For any $p$-complete animated $\overline{A}$ algebra $R$, the map
	\[ comp_{R/A}\colon L\Prism_{R/A} \to \Prism_{R/A} \]
is an equivalence.

\end{theorem}

The proof of Theorem \ref{main comparison} will take some preliminaries. We begin by establishing the theorem in a special case:

\begin{lemma} Suppose $R = A/J$ where $J = (I,x_{1}, ... , x_{n})$ is generated by a $(p,I)$-completely regular sequence. Then $comp_{R/A}$ is an equivalence.

\end{lemma}

\begin{proof} Combine the description $L\Prism_{R/A} \simeq \compenv_{I}(Free^{\delta}_{A}(F^{\star}_{H} \mathbbl{\Pi}_{R/A}))$ with Example \ref{infinitesimal of regular sequence} and Observation \ref{envelopes of 01 filtered} to identify $L\Prism_{R/A}$ with
	\[ L\Prism_{R/A} \simeq \compenv_{I}(J \to A). \]
Corollary \ref{envelopes of regular sequences} then yields $L\Prism_{R/A} \simeq A\{\frac{J}{I}\}^{\wedge}_{(p,I)}$. The result then follows from Example 7.9 in \cite{Bhatt-Scholze} which identifies the derived prismatic cohomology of $R$ over $A$ with the prismatic envelope.
\end{proof}

Prismatic cohomology (of smooth algebras) can be computed via an appropriate \v{C}ech--Alexander complex, where each of the terms in the complex fits into the context of the preceding lemma. It therefore suffices to prove that $L\Prism_{R/A}$ can be accessed via \v{C}ech--Alexander complexes in the same way. To establish this, we will use the factorization $L\Prism_{R/A} = Free^{\delta}_{A} \circ L\Theta_{R/A}$ to break the problem up into two stages. We will begin by proving that $L\Theta_{R/A}$ can be computed via a \v{C}ech--Alexander complex (see Proposition \ref{descent for theta} below), and then we will appeal to the filtration on $Free^{\delta}_{A}$ by polynomial subfunctors from Section 2 to deduce the analogous claim for $L\Prism_{R/A}$.

\begin{lemma} \label{graded theta} Denote by $L\overline{\Theta}_{-/A} := L\Theta_{-/A} \otimes_{A} \overline{A}$. Then there is a canonical factorization

\begin{center}
\begin{tikzcd}

\calg(\Mod_{\overline{A}}) \arrow[rr, "L\overline{\Theta}_{-/A}"] \arrow[ddr, "gr^{\star}_{H}\mathbbl{\Pi}_{R/A}"]
	&& \compcalg(\Mod_{\overline{A}}) \\ \\
& \calg(\gr \Mod_{gr^{\star}(I^{\star}A)}) \arrow[uur, "\grenv_{I}"]

\end{tikzcd}
\end{center}
\end{lemma}

\begin{proof} 

Recall from Variant \ref{graded envelope}, the base change functor $ins^{0}(-) \otimes_{\overline{A}(0)} gr^{\star}(I^{\star}A)$ admits a left adjoint given by the graded $I$-adic envelope $\grenv_{I}$. We appeal to Lemma \ref{graded vs filtered envelopes} to conclude:
\begin{align*}
	   \grenv_{I}(gr^{\star}_{H} \mathbbl{\Pi}_{-/A}) & \simeq \calgenv_{I}(F^{\star}_{H} \mathbbl{\Pi}_{-/A}) \otimes_{A} \overline{A} \\
				& \simeq L \overline{\Theta}_{-/A}.
\end{align*}
\end{proof}

\begin{construction} \label{conjugate filtration for theta} (The Conjugate Filtration on $L\overline{\Theta}$) Recall from Theorem \ref{vertical filtration} that we can endow $gr^{\star}_{H} \mathbbl{\Pi}_{R/A}$ with a functorial exhaustive increasing filtration, denoted by $F^{v}_{\star} gr^{\star}_{H} \mathbbl{\Pi}_{R/A}$, and thus view it as an object of the $\infty$-category $   \calg(F_{\geq 0} \gr \Mod_{gr^{\star}(I^{\star}A)})$. To be clear, we are viewing $gr^{\star}(I^{\star}A)$ as a filtered graded object concentrated in filtration weight 0 (so the relevant derived algebraic context is $ F_{\geq 0} \gr \Mod_{A}  := Fun(\Z_{\geq 0} \times \Z_{\geq 0}^{disc}, \Mod_{A})$). 

We define the \emph{conjugate filtration} on $L \overline{\Theta}_{-/A}$ to be the composite
	\[ \calg(\Mod_{\overline{A}}) \xrightarrow{ F^{v}_{ \star} gr^{\star}_{H} \mathbbl{\Pi}_{-/A}} \calg(\gr F_{\geq 0} \Mod_{A})_{gr^{\star}(I^{\star}A)} \xrightarrow{\grenv_{I}} \calg(F_{\geq 0} \Mod_{A}) \]	
which we denote by $F^{conj}_{\star} L \overline{\Theta}_{-/A}$. A simple unwinding of definitions combined with Lemma \ref{graded theta} identifies $und(F^{conj}_{\star} L \overline{\Theta}_{R/A}) \simeq L \overline{\Theta}_{R/A}$. We now turn our attention towards identifying the associated graded.
\end{construction}

\begin{theorem} There is a functorial isomorphism of graded algebras
	\[ gr_{\star}^{conj} L \overline{\Theta}_{R/A} \simeq LSym^{\star}_{R}(L_{R/\overline{A}}[-1]\{-1\}(1)). \]
\end{theorem}

\begin{proof} We know from Proposition \ref{naturality for envelopes} that we can express the left hand side as \begin{align*}
gr_{\star}^{conj} L \overline{\Theta}_{R/A} & \simeq gr_{\star}^{conj} \grenv_{I} \circ F_{\star}^{v}(gr^{\star}_{H} \mathbbl{\Pi}_{R/A}) \\
								& \simeq \grenv_{I} gr_{\star}^{vert} (gr^{\star}_{H} \mathbbl{\Pi}_{R/A})
\end{align*} 
Now appealing to Theorem \ref{vertical filtration}, we can explicitly identify the associated graded of the vertical filtration and thus rewrite the last line as
\[ \grenv_{I} (LSym_{R}^{\star}(L_{R/\overline{A}}[-1](1)) \otimes_{\overline{A}} LSym_{A}(I/I^{2}(1))) \simeq LSym_{R}(L_{R/\overline{A}}[-1]\{-1\}(1))\]
where the last equivalence follows from Lemma \ref{graded envelope of free algebras}.
\end{proof}

\begin{proposition} \label{descent for theta} Let $R$ be a smooth $A/I$-algebra. Denote by $Poly_{A \downarrow R}$ the $1$-category consisting of $(p,I)$-complete polynomial $A$-algebras $P$ equipped with a map $\alpha\colon P \to R$. Then the canonical map
	\[ L\Theta_{R/A} \to \text{lim}_{Poly_{A \downarrow R}} L\Theta_{R/P} \]
is an equivalence.
\end{proposition}

\begin{proof}

It suffices to check that the map
	\[ L\Theta_{R/A} \to \text{lim}_{Poly_{A \downarrow R}} L\Theta_{R/P} \]
is an equivalence after reduction modulo $I$, as both sides are $I$-adically complete. Since $I \subset A$ is a Cartier divisor, reduction modulo $I$ preserves the totalization on the right hand side, reducing us to verifying that the map
	\[ L\overline{\Theta}_{R/A} \to \text{lim}_{Poly_{A \downarrow{R}}} L \overline{\Theta}_{R/P} \]
is an equivalence.

Since $R$ is assumed to be $p$-completely smooth, we may choose a surjection $P \to R$ in $Poly_{A \downarrow R}$, and restrict our attention to the cofinal subdiagram given by the \v{C}ech nerve
	\[ L \overline{\Theta}_{R/A} \to \text{lim}_{\Delta} L \overline{\Theta}_{R/P^{\otimes \cdot + 1}}. \]
	
At this point, we can appeal to the conjugate filtration to reduce to showing that $LSym^{i}_{R}(L_{R/A}[-1])$ is computed by the relevant totalization.
The cotangent complex satisfies the desired descent properties by combining Proposition 3.1 of \cite{BMS2} with the ($p$-completed) conormal sequence of the cotangent complex. Furthermore, since $P \to R$ is a surjection whose kernel is generated by a $p$-completely regular sequence (since $R$ is $p$-completely smooth, one can always choose the surjection to satisfy this), $L^{\wedge}_{R/P}\text{[-1]}$ is in fact a finitely presented projective $R$-module. Hence $LSym_{R}$ preserves the desired totalization, as it is the left-right extension of its restriction to the finitely presented projective $R$-modules.
\end{proof}

\begin{theorem} \label{prismatic Cech-Alexander} Denote by $Poly_{A \downarrow R}$ the $1$-category consisting of $(p,I)$-complete polynomial $A$-algebras $P$ equipped with a map $\alpha\colon P \to R$. Let $F_{P} := Free^{\delta}_{A}(P)$ and $S_{P} := F_{P} \otimes_{P} R$. Then the canonical map
	\[ L\Prism_{R/A} \to \text{lim}_{Poly_{A \downarrow R}} L\Prism_{S_{P}/F_{P}} \]
is an equivalence.

\end{theorem}

\begin{proof} Recall that $L\Prism_{R/A} \simeq Free^{\delta}_{A} \circ L\Theta_{R/A}$, and observe that $Free^{\delta}_{A} \circ L\Theta_{R/P} \simeq Free^{\delta}_{A} \circ L\Theta_{S_{P}/F_{P}}$. Appealing to Proposition \ref{descent for theta}, we are thus reduced to verifying that
	\[ Free^{\delta}_{A}(\text{lim}_{Poly_{A \downarrow R}} L\Theta_{R/P}) \simeq \text{lim}_{Poly_{A \downarrow R}}(L \Prism_{S_{P}/F_{P}}). \]
Fix $P \in Poly_{A \downarrow R}$ such that $P \to R$ is a surjection and the kernel is generated by a regular sequence $(f_{1},..., f_{k})$ (this is always possible locally on $R$, and the claim in question is local). We then obtain an identification $F^{0} := L\Theta_{R/P} \simeq A\left<x_{1},..., x_{n}\right>\left[ \frac{f_{i}}{I}\right]^{\wedge}_{(p,I)}$ via the same argument as Corollary \ref{envelopes of regular sequences}. In particular, $F^{0}$ and all terms in the \v{C}ech nerve of $A \to F^{0}$ are discrete, since $F^{0}$ is $(p,I)$-completely flat over $A$. We will denote the \v{C}ech nerve by $F^{\star}$.  It suffices to show that $Free^{\delta}_{A}$ preserves the totalization of $F^{\star}$ - i.e. the canonical map
	\[ Free^{\delta}_{A}(Tot(F^{\star})) \to Tot(Free^{\delta}_{A}(F^{\star})) \]
is an equivalence. 

Since each $F^{n}$ is discrete, we may view $F^{\star}$ as a cosimplicial commutative ring (as opposed to a cosimplicial derived commutative ring), and since the inclusion
	\[ CAlg_{A}^{\heartsuit} \to \calg_{A} \]
preserves filtered colimits, we can pass to the cosimplicial skeleta to rewrite
	\[ F^{\star} \simeq colim_{n} sk_{n}F^{\star} \]
where the skeleta are taken in cosimplicial $A$-algebras, and so in particular remain pointwise discrete.
We now claim that we can commute the filtered colimit and the totalization:
	\[ Tot(F^{\star}) \simeq Tot(colim_{n} sk_{n}(F^{\star})) \simeq colim_{n} Tot(sk_{n}(F^{\star})). \]
Indeed, since the forgetful functor $CAlg(\Mod_{A}) \to \Mod_{A}$ is conservative and preserves filtered colimits and all limits, it suffices to check on the level of underlying $A$-modules. Since each term in the cosimplicial object is discrete, the $m^{th}$ truncation (as a cosimplicial $A$-module) $\tau_{\geq m}(Tot(F^{\star}))$ depends only on the $m^{th}$-coskeleton of the appearing totalizations, which reduces us to the case of commuting a filtered colimit with a finite limit.

Recall that $Free^{\delta}_{A} \circ Sym^{\heartsuit}_{A} = Sym^{\delta, \heartsuit}_{A}$ admits an exhaustive filtration by excisively polynomial subfunctors, and thus $Free^{\delta}_{A}$ is the non-linear right-left extension of its restriction to polynomial $A$-algebras. In particular, Theorem \ref{non-linear techy extension result} guarantees that $Free^{\delta}_{A}$ preserves finite totalizations of $A$-algebras. Hence we obtain
\begin{align*}
Free^{\delta}_{A}(Tot(F^{\star})) & \simeq colim_{n} Free^{\delta}_{A}(Tot(sk_{n}(F^{\star}))) \\
	& \simeq colim_{n} Tot(Free^{\delta}_{A}(sk_{n}(F^{\star})) \\
	& \simeq Tot(Free^{\delta}_{A}(F^{\star})) 
\end{align*}
as desired.
\end{proof}

\begin{proof}[Proof of Theorem \ref{main comparison}] It suffices to verify the theorem in the case that $R$ is a $p$-completely smooth $\overline{A}$-algebra, since both $L\Prism_{-/A}$ and $\Prism_{-/A}$ preserve sifted colimits (as always, $\Prism_{-/A}$ is viewed as the left Kan extension of the site-theoretic cohomology from the smooth case). 

The comparison map
	\[ comp_{R/A} \colon L\Prism_{R/A} \to \Prism_{R/A} \]
is functorial both in $R$ and in $A$. In particular, given any $(p,I)$-complete polynomial $A$-algebra $P$ with a surjection
	\[ P \to R \]
$comp_{R/A}$ refines to a map of \v{C}ech--Alexander complexes
	\[comp_{S^{\star}_{P}/F^{\star}_{P}} \colon L\Prism_{S^{\star}_{P}/F^{\star}_{P}} \to \Prism_{S^{\star}_{P}/F^{\star}_{P}} \]
which is a termwise equivalence by Corollary \ref{envelopes of regular sequences}.

We know from Construction 4.18 in \cite{Bhatt-Scholze} that
	\[ \Prism_{R/A} \simeq lim_{\Delta} \Prism_{S^{\star}_{P}/F^{\star}_{P}} \]
from which we conclude.
\end{proof}

As an application of these definitions, we can establish affineness of the relative prismatization of an affine formal scheme. Recall, for any derived $p$-adic formal $\overline{A}$-scheme $X$, the relative prismatization is the presheaf on $(p,I)$-nilpotent animated $A$-algebras defined by
	\[ WCart_{X/A}(B) = X(\overline{W(B)}). \]
Together with Theorem \ref{main comparison}, the following Lemma recovers Theorem 7.17 in \cite{Bhatt-Lurie2}.

\begin{lemma} For any $p$-complete animated commutative $\overline{A}$-algebra $R$, there is a canonical equivalence
	\[ WCart_{Spf(R)/A} \simeq Spf(L\Prism_{R/A}). \]

\end{lemma}

\begin{proof} Recall, the derived Witt vectors functor (Definition \ref{derived Witt vectors}) is right adjoint to the forgetful functor $\dalg_{\Z_{p}} \to \calg_{\Z_{p}}$, and on connective objects this recovers the animated Witt vectors (Observation \ref{animated vs derived Witt vectors}). It follows from Lemma \ref{adjunctions on slice cats} that the forgetful functor $\dalg_{A} \to \calg_{A}$ also admits a right adjoint which is once again given by derived Witt vectors.	
Unwinding the various universal properties in play, we see
\begin{align*}
 Spf(L\Prism_{R/A})(S) &\simeq Map_{\calg_{A}}(L\Prism_{R/A}, S) \\
 	& \simeq Map_{\dalg_{A}}(L\Prism_{R/A}, W(S)) \\
	& \simeq Map_{\calg_{\overline{A}}}(R, \overline{W(S)}) \\
	& \simeq WCart_{Spf(R)/A}(S)
\end{align*}
from which we conclude.
\end{proof}

\addcontentsline{toc}{section}{References}
\bibliography{references}

\begin{bibdiv}
\begin{biblist}

\bib{Antieau}{misc}{
      author={Antieau, Benjamin},
       title={Filtrations and cohomology},
   publisher={to appear},
        date={2023},
         url={to appear},
}

\bib{Bhatt-Lurie}{misc}{
      author={Bhatt, Bhargav},
      author={Lurie, Jacob},
       title={Absolute prismatic cohomology},
   publisher={arXiv preprint arXiv.2201.06120},
        date={2022},
         url={https://arxiv.org/abs/2201.06120},
}

\bib{Bhatt-Lurie2}{misc}{
      author={Bhatt, Bhargav},
      author={Lurie, Jacob},
       title={The prismatization of $p$-adic formal schemes},
   publisher={arXiv preprint arXiv.2201.06124},
        date={2022},
         url={https://arxiv.org/abs/2201.06124},
}

\bib{BMS2}{article}{
      author={Bhatt, Bhargav},
      author={Morrow, Matthew},
      author={Scholze, Peter},
       title={Topological {H}ochschild homology and integral {$p$}-adic {H}odge
  theory},
        date={2019},
        ISSN={0073-8301},
     journal={Publ. Math. Inst. Hautes \'{E}tudes Sci.},
      volume={129},
       pages={199\ndash 310},
         url={https://doi.org/10.1007/s10240-019-00106-9},
}

\bib{Bhatt-Scholze2}{article}{
      author={Bhatt, Bhargav},
      author={Scholze, Peter},
       title={Projectivity of the {W}itt vector affine {G}rassmannian},
        date={2017},
        ISSN={0020-9910},
     journal={Invent. Math.},
      volume={209},
      number={2},
       pages={329\ndash 423},
         url={https://doi.org/10.1007/s00222-016-0710-4},
}

\bib{Bhatt-Scholze}{article}{
      author={Bhatt, Bhargav},
      author={Scholze, Peter},
       title={{Prisms and prismatic cohomology}},
        date={2022},
     journal={Annals of Mathematics},
      volume={196},
      number={3},
       pages={1135 \ndash  1275},
         url={https://doi.org/10.4007/annals.2022.196.3.5},
}

\bib{Brantner-Mathew}{misc}{
      author={Brantner, Lukas},
      author={Mathew, Akhil},
       title={Deformation theory and partition {L}ie algebras},
   publisher={arXiv preprint arXiv.1904.07352},
        date={2019},
         url={https://arxiv.org/abs/1904.07352},
}

\bib{Devalapurkar-Haine}{article}{
      author={Devalapurkar, Sanath},
      author={Haine, Peter},
       title={On the {J}ames and {H}ilton-{M}ilnor splittings, and the
  metastable {EHP} sequence},
        date={2021},
        ISSN={1431-0635},
     journal={Doc. Math.},
      volume={26},
       pages={1423\ndash 1464},
      review={\MR{4334846}},
}

\bib{Goodwillie}{article}{
      author={Goodwillie, Thomas~G.},
       title={Calculus. {III}. {T}aylor series},
        date={2003},
        ISSN={1465-3060},
     journal={Geom. Topol.},
      volume={7},
       pages={645\ndash 711},
         url={https://doi.org/10.2140/gt.2003.7.645},
}

\bib{Guo-Li}{misc}{
      author={Guo, Haoyang},
      author={Li, Shizhang},
       title={Period sheaves via derived de {R}ham cohomology},
   publisher={arXiv preprint arXiv.2008.06143},
        date={2020},
         url={https://arxiv.org/abs/2008.06143},
}

\bib{Li-Liu}{misc}{
      author={Li, Shizhang},
      author={Liu, Tong},
       title={Comparison of prismatic cohomology and derived de {R}ham
  cohomology},
   publisher={arXiv preprint arXiv.2012.14064},
        date={2020},
         url={https://arxiv.org/abs/2012.14064},
}

\bib{Lurie}{book}{
      author={Lurie, Jacob},
       title={Higher topos theory},
      series={Annals of mathematics studies 170},
   publisher={Princeton University Press},
        date={2012},
}

\bib{LurieHA}{misc}{
      author={Lurie, Jacob},
       title={Higher algebra},
   publisher={URL: https://www.math.ias.edu/~lurie/papers/HA.pdf},
        date={2017},
}

\bib{kerodon}{misc}{
      author={Lurie, Jacob},
       title={Kerodon},
   publisher={URL: https://kerodon.net},
        date={2022},
}

\bib{Raksit}{misc}{
      author={Raksit, Arpon},
       title={Hochschild homology and the derived de {R}ham complex revisited},
   publisher={arXiv preprint arXiv.2007.02576},
        date={2020},
         url={https://arxiv.org/abs/2007.02576},
}

\end{biblist}
\end{bibdiv}

\end{document}